\documentclass[a4paper,11pt]{article}
\usepackage[T2A]{fontenc}
\usepackage[utf8]{inputenc}
\usepackage[russian,english]{babel}

\usepackage{latexsym}
\usepackage{geometry}
\usepackage{graphicx}
\usepackage{amsfonts}
\usepackage{amsthm}
\usepackage{amsmath}
\usepackage{amssymb}
\usepackage{xcolor}
\usepackage{mathrsfs}
\usepackage{bookmark}
\usepackage{hyperref}
\usepackage[noadjust]{cite}

\usepackage{accents}
\usepackage{indentfirst}
\usepackage{enumerate}
\usepackage{mathtools}
\usepackage{csquotes}
\usepackage{setspace}
\usepackage{cite}
\usepackage{array}
\usepackage{float}
\usepackage{arydshln}
\usepackage{pdflscape}

\usepackage{longtable}
\usepackage{multirow}
\usepackage{multicol}
\usepackage{placeins}
\usepackage{makecell}

\usepackage{enumitem}
\setlist[enumerate]{font=\upshape,label=(\arabic*)}

\newtheorem{theorem}{Theorem}[section]
\newtheorem{corollary}[theorem]{Corollary}
\newtheorem{proposition}[theorem]{Proposition}
\newtheorem{lemma}[theorem]{Lemma}
\theoremstyle{definition}
\newtheorem{definition}[theorem]{Definition}
\newtheorem{descript}[theorem]{Description}

\newtheorem{notation}[theorem]{Notation}
\newtheorem{example}[theorem]{Example}
\newtheorem{remark}[theorem]{Remark}

\numberwithin{equation}{section}
\allowdisplaybreaks

\let\phi\varphi

\renewcommand{\Re}{\mathop{\mathfrak{Re}}\nolimits}
\renewcommand{\Im}{\mathop{\mathfrak{Im}}\nolimits}

\def\A{\mathcal{A}}
\def\E{\mathcal{E}}
\def\Hyp{\mathcal{H}}
\def\Main{\mathcal{M}}

\DeclareMathOperator{\chrs}{char}

\DeclareMathOperator{\sgn}{sgn}

\DeclareMathOperator{\Lin}{Lin}

\newcommand\til[1]{\widetilde{\phantom{,\mkern-4mu} #1 \phantom{,\mkern-4mu}}\mkern-1mu}

\newcommand\numeq[1]%
  {\stackrel{\scriptstyle(\mkern-1.5mu#1\mkern-1.5mu)}{=}}

\newcounter{rowcntr}[table]
\renewcommand{\therowcntr}{3.20.\arabic{rowcntr}}
\newcolumntype{N}{>{\refstepcounter{rowcntr}\therowcntr}c}
\AtBeginEnvironment{tabular}{\setcounter{rowcntr}{0}}

\providecommand{\keywords}[1]{\textbf{Keywords:} #1}
\providecommand{\msc}[1]{\textbf{MSC 2020:} #1}

\newcommand\freefootnote[1]{%
    \bgroup
    \renewcommand\thefootnote{\fnsymbol{footnote}}%
    \renewcommand\thempfootnote{\fnsymbol{mpfootnote}}%
    \footnotetext[0]{#1}%
    \egroup
}

\begin{document}

\title{Orthogonality graphs of real Cayley--Dickson algebras. Part II: The subgraph on pairs of basis elements}
\author{
Svetlana Zhilina$^{a,b}$
}
\date{\small \em
$^a$Department of Mathematics and Mechanics, Lomonosov Moscow State\\ University, Moscow, 119991, Russia\\
$^b$Moscow Center of Fundamental and Applied Mathematics, Moscow, 119991, Russia
}

\maketitle

\begin{abstract}
We consider zero divisors of an arbitrary real Cayley--Dickson algebra such that their components are both standard basis elements. We construct inductively the orthogonality graph on these elements. Then we show that, if we restrict our attention to at least $16$-dimensional algebras, two algebras are isomorphic if and only if their graphs are isomorphic. We also provide an algorithm to retrieve the Cayley--Dickson parameters of an algebra from its graph.
\end{abstract}

\keywords{Cayley--Dickson algebras, standard basis elements, orthogonality graphs, graph isomorphism, algebra isomorphism.}

\msc{05C25, 17A20, 17D05, 17D99.}

\freefootnote{This work was supported by the Russian Science Foundation (project No. 17-11-01124).}

\freefootnote{Email address: \texttt{s.a.zhilina@gmail.com}}

\section{Introduction}

Cayley--Dickson algebras are a family of $2^n$-dimensional algebras $\A_n$ over $\mathbb{R}$, $n \in \mathbb{N}_0$, which generalize the real numbers, the complex numbers, the quaternions, and the octonions. These four algebras possess a pleasant property of alternativity, however, it is well known that Cayley--Dickson algebras stop being alternative for $n \geq 4$. One of the consequences is the existence of zero divisors in the sedenions and other higher-dimensional algebras, and the problem of their classification is not a trivial one.

At present most of the authors restrict their attention to the algebras of the main sequence which we denote by~$\Main_n$. In these algebras all parameters which determine the Cayley--Dickson construction are assumed to be equal to $-1$. The most successful efforts have been taken by Moreno~\cite{moreno, moreno_constructing} and Biss, Dugger, and Isaksen~\cite{biss, biss2}. Moreno's key idea was to study doubly alternative zero divisors, that is, such elements that their components are both alternative elements of the previous algebra. This notion was extended in~\cite{our_split-algebras} to the Cayley--Dickson split-algebras, denoted by~$\Hyp_n$, and led to similar results.

A particular case of doubly alternative zero divisors are the elements whose components are both standard basis elements up to sign, that is, which are of the form $(e_i, \pm e_j)$. In some higher-dimensional ($n = 4,5,6$) algebras of the main sequence such zero divisors have been studied by de Marrais, see~\cite{marrais, marrais2, marrais3}.

We are also concerned with a convenient representation of relations between zero divisors. One of the possible solutions is to draw a zero divisor graph or an orthogonality graph of the algebra. Actually, studying relation graphs of various algebras has become a rapidly extending branch of mathematics.

Relation graphs of some low-dimensional Cayley--Dickson algebras have been recently described, see~\cite{our_sedenions} for the sedenions, and~\cite{our_split-algebras, our_split-sedenions} for the split-complex numbers, the split-quaternions, the split-octonions, and the split-sedenions. The proof relies vastly on the fact that for $n \leq 4$ all elements of $\A_n$ are doubly alternative. However, when the dimension reaches $32$, the graphs become almost impossible to study.

Still, we desire to distinguish different Cayley--Dickson algebras by their relation graphs. For this purpose, we take a part of the orthogonality graph which is easier to consider, namely, the subgraph on the vertices of the form $(e_i, \pm e_j)$ with $i \neq 0$. We denote it by $\Gamma_e(\A_n)$. In Section~\ref{section:another-approach} we describe another approach to the choice of a convenient subgraph.

This paper is a sequel to our previous work~\cite{our_orthographs1}, which explores zero divisors in arbitrary real Cayley--Dickson algebras whose components satisfy some additional conditions on the norm and alternativity. We have discovered that they form hexagonal patterns in orthogonality and zero divisor graphs, see~\cite[Corollary~3.7]{our_orthographs1}. In the algebras of the main sequence, any pair of zero divisors generates the so-called double hexagon, see~\cite[Description~4.8]{our_orthographs1}, whose vertices have a convenient multiplication table, cf.~\cite[Theorem~4.11]{our_orthographs1}. As a preparation for this article, we have also determined orthogonality conditions for zero divisors of the form $(e_i, \pm e_j)$. These conditions are summarized in Theorem~\ref{theorem:orthogonality-conditions}.

Vastly based on the results mentioned, in Theorem~\ref{theorem:nonspecial-components} we construct inductively the graph $\Gamma_e(\A_{n+1})$ for an arbitrary real Cayley--Dickson algebra $\A_{n+1}$. We prove also that for $n \geq 3$ the graph contains either $2^n$ or $2^n - 1$ connected components, depending on the last Cayley--Dickson parameter $\gamma_n$, and the diameter of each connected component equals three, see Corollary~\ref{corollary:diameters}. 

Section~\ref{section:retrieving-parameters} is devoted to the converse problem, namely, we retrieve all Cayley--Dickson parameters of $\A_{n+1}$ from $\Gamma_e(\A_{n+1})$. We solve this problem in several steps:
\begin{itemize}
    \item Corollary~\ref{corollary:diameters} provides an immediate method to obtain $\gamma_n$ for $n \geq 3$.
    \item If $n \geq 4$, then, by Lemma~\ref{lemma:find-special-components}, we can also determine $\gamma_{n-1}$.
    \item We then use Theorem~\ref{theorem:construct-graphs-recursively} and reduce $\Gamma_e(\A_{n+1})$ into graphs $\Gamma_e(\A^{\circ}_n)$ and $\Gamma_e(\A^{\bullet}_n)$, where $\A^{\circ}_n$ and $\A^{\bullet}_n$ have the same first $n-1$ parameters as $\A_{n+1}$, and their last parameters are equal to $\gamma_n$ and $\gamma_{n-1} \gamma_n$, respectively.
    \item Proceeding recursively, we obtain Corollary~\ref{corollary:isomorphic-algebras} and Theorem~\ref{theorem:isomorphic-graphs} which state that for $n \geq 3$ two algebras are isomorphic if and only if their graphs are isomorphic. 
\end{itemize}

The crucial ingredient in our last step is a work by Eakin and Sathaye~\cite{eakin} on Cayley--Dickson algebras automorphisms.

\section{An overview of real Cayley--Dickson algebras} \label{section:A_n}

\subsection{Algebraic relations and their graphs} \label{subsection:definitions}

Let~$\mathbb{F}$ be an arbitrary field and $(\A, +, \cdot)$ be an algebra with a unity $1_{\A}$ over the field~$\mathbb{F}$. $\A$ is not assumed to be commutative or associative.
We denote the set of zero divisors (left, right, or two-sided) of $\A$ by $Z(\A)$, and the set of two-sided zero divisors of $\A$ by $Z_{LR}(\A)$. However, by~\cite[Corollary~4.6]{our_split-algebras}, in case of real Cayley--Dickson algebras these two sets coincide, that is, $Z(\A_n) = Z_{LR}(\A_n)$.

We say that $a$ and $b$ in $\A$ {\em are orthogonal} if $ab = ba = 0$. The orthogonalizer of $a \in \A$, that is, the set of all elements which are orthogonal to $a$, is denoted by $O_\A(a)$. Clearly, $O_\A(a)$ is a vector space over~$\mathbb{F}$.

\begin{notation}
For any set $X \subseteq \A$ we denote the set of lines passing through elements of $X$ by
$$
P(X) = \{ [x] = \mathbb{F} x \; | \; x \in X \}.
$$
\end{notation}

\begin{definition}
Let $\A$ be an arbitrary algebra.
\begin{itemize}
\item 
{\em The orthogonality graph} $\Gamma_O(\A)$ is defined as follows: its vertices are lines in $Z_{LR}(\A)$, that is,
$$
V(\Gamma_O(\A)) = P(Z_{LR}(\A)),
$$
and distinct vertices $[a]$ and $[b]$ are adjacent if and only if $ab=ba=0$.
\item
{\em The directed zero divisor graph} $\Gamma_Z(\A)$ is defined as follows: its vertices are lines in $Z(\A)$, that is,
$$
V(\Gamma_Z(\A)) = P(Z(\A)),
$$
and distinct vertices $[a]$ and $[b]$ form a directed edge $([a], [b])$ if and only if $ab=0$.
\end{itemize}
\end{definition}

Note that the edges of $\Gamma_O(\A)$ and $\Gamma_Z(\A)$ are well-defined. When speaking of the vertices of these graphs, we will not distinguish between a nonzero element $a$ and a line $[a] = \mathbb{F} a$ passing through it.

In an undirected graph $\Gamma$, $d(x,y) = d_{\Gamma}(x,y)$ denotes the distance between two vertices $x$ and $y$, and $d(\Gamma) = \sup\limits_{x,y \in \Gamma} d(x,y)$ denotes the diameter of $\Gamma$.

\subsection{Constructing Cayley--Dickson algebras} \label{subsection:A_n}

We refer the reader to~\cite{mccrimmon,schafer} for auxiliary definitions and general properties of Cayley--Dickson algebras.

\begin{definition} \label{definition:Cayley--Dickson-algebras}
Let $\A$ be an algebra over a field~$\mathbb{F}$ with an involution $a \mapsto \bar{a}$. The algebra $\A \{ \gamma \}$ produced by the Cayley--Dickson process, when applied to $\A$ with the parameter $\gamma \in \mathbb{F}$, $\gamma \neq 0$, is defined as the set of ordered pairs of elements of $\A$ with operations
\begin{align*}
\alpha(a,b)&=(\alpha a, \alpha b);\\
(a,b)+(c,d)&=(a+c,b+d);\\
(a,b)(c,d)&=(ac+\gamma \bar{d}b,da+b\bar{c})
\end{align*} 
and the involution
$$
\qquad (\overline{a,b})=(\bar{a},-b), \qquad a,b,c,d\in \A, \ \alpha \in \mathbb{F}.
$$
If the involution on $\A$ is regular, that is, $a + \bar{a} \in \mathbb{F}1_{\A}$ and $a\bar{a} = \bar{a}a \in \mathbb{F}1_{\A}$ for all $a \in \A$, then the involution on $\A \{ \gamma \}$ is also regular, cf.~\cite[p.~435]{schafer}.
\end{definition}

Henceforth we assume that $\mathbb{F} = \mathbb{R}$. We now define an arbitrary real Cayley--Dickson algebra which is determined by the set of its parameters. The most common definition of real Cayley--Dickson algebras assumes that all parameters are equal to $-1$. These particular algebras are what we call the algebras of the main sequence.

\begin{definition} \label{definition:A_n}
For every integer $n \geq 0$ and nonzero real numbers $\gamma_0, \dots, \gamma_{n-1}$ we define the real Cayley--Dickson algebra $\A_n = \A_n \{ \gamma_0, \dots, \gamma_{n-1} \}$ inductively:
\begin{enumerate} 
\item $\A_0 = \mathbb{R}$, and $e^{(0)}_0 = 1$ is its only basis element;
\item If $\A_n \{ \gamma_0, \dots, \gamma_{n-1} \}$ is constructed, then $\A_{n+1} \{ \gamma_0, \dots, \gamma_n \} =
\\
(\A_n \{ \gamma_0, \dots, \gamma_{n-1} \}) \{ \gamma_n \}$. Its basis elements are $e^{(n+1)}_0, \dots, e^{(n+1)}_{2^{n+1}-1}$ such that
\begin{equation*}
e^{(n+1)}_j =
\begin{cases}
(e^{(n)}_j,0), & 0 \leq j \leq 2^n-1,\\
(0,e^{(n)}_{j-2^n}), & 2^n \leq j \leq 2^{n+1}-1.
\end{cases}
\end{equation*}
\end{enumerate}
\end{definition}

For every integer $n \geq 0$ the structure $\A_n$ in Definition~\ref{definition:A_n} is a $2^n$-dimensional algebra over $\mathbb{R}$ with the unit element $e^{(n)}_0$ and a regular involution, cf.~\cite[Lemma 3.14]{our_split-algebras}.
Consider the following definitions which are analogous to those for complex numbers.

\begin{definition} \label{definition:real-imaginary-part}
\leavevmode
\begin{itemize}
	\item Let $a \in \A_n$. Its {\em real part} is $\Re(a) = \frac{a + \bar{a}}{2}$, its {\em imaginary part} is $\Im(a) = \frac{a - \bar{a}}{2}$, and its {\em norm} is $n(a) = a \bar{a} = \bar{a}a$. Since the involution on $\A_n$ is regular, $\Re(a), n(a) \in \mathbb{R}$.
	\item An element $a \in \A_n$ is said to be {\em pure} if $\Re(a)=0$.
	\item An element $(a, b) \in \A_{n+1}$ is said to be {\em doubly pure} if $\Re(a) = \Re(b) = 0$.
\end{itemize}
\end{definition}

\begin{proposition}[{\cite[p. 435]{schafer}}]
We can compute real part and norm of an element $(a,b) \in \A_{n+1}$ inductively by using the following equalities:
\begin{align*}
	\Re((a,b)) &= \Re(a),\\
	n((a,b)) &= n(a) - \gamma_n n(b).
\end{align*}
\end{proposition}

\begin{notation}
We denote
\begin{align*}
    \E_n &= \left\{ e^{(n)}_0, e^{(n)}_1, \dots, e^{(n)}_{2^{n}-1} \right\},\\
    \E'_n &= \left\{ e^{(n)}_1, \dots, e^{(n)}_{2^{n}-1} \right\} = \E_n \setminus \left\{ e^{(n)}_0 \right\},\\
    \E''_n &= \left\{ e^{(n)}_1, \dots, e^{(n)}_{2^{n-1}-1}, e^{(n)}_{2^{n-1}+1}, \dots e^{(n)}_{2^{n}-1} \right\} = \E_n \setminus \left\{ e^{(n)}_0, e^{(n)}_{2^{n-1}} \right\}.
\end{align*}
Clearly, $\E'_n$ is the set of pure basis elements, and $\E''_n$ is the set of doubly pure basis elements. We also denote
$$
\A'_n = \{ a \in \A_n \; | \; \Re(a) = 0 \}.
$$
\end{notation}

\subsection{Some properties of real Cayley--Dickson algebras} \label{subsection:A_n-properties}

Henceforth we assume that $\A$ is an arbitrary algebra over a field~$\mathbb{F}$, and $\A_n = \A_n \{ \gamma_0, \dots, \gamma_{n-1} \}$ is an arbitrary real Cayley--Dickson algebra. By~\cite[Exercise 2.5.1]{mccrimmon}, $\A_n \{ \gamma_0, \dots, \gamma_{n-1} \}$ is isomorphic to $\A_n \{ \sgn(\gamma_0), \dots, \sgn(\gamma_{n-1}) \}$, so it is sufficient to consider only $\gamma_k \in \{ \pm 1 \}$, $k = 0, \dots, n-1$.

The associator of $a,b,c \in \A$ is defined as the element $[a,b,c]=(ab)c-a(bc)$. An algebra $\A$ is called {\em flexible} if for all $a,b \in \A$ the equality $[a,b,a] = 0$ holds. Clearly, in a flexible algebra $\A$, we have $[a,b,c]=-[c,b,a]$ for all $a,b,c \in \A$. An algebra $\A$ is called {\em alternative} if for all $a,b \in \A$ the equalities $[a,a,b] = [b,a,a] = 0$ hold.

It is well known that $\A_n$ is alternative if and only if $n \leq 3$, however, $\A_n$ is always flexible, see, e.g.,~\cite[p. 436, Theorem~1]{schafer}.

\begin{proposition} \label{proposition:lambda-form}
Let $\langle a, b \rangle$ denote a real-valued symmetric bilinear form associated with the quadratic form $n(a)$. Then $\langle a, a \rangle = n(a)$ and $2\langle a,b \rangle = a \bar{b} + b \bar{a} = \bar{a} b + \bar{b} a$ for all $a, b \in \A_n$, cf. \cite[Proposition~3.18, Proposition~3.19]{our_split-algebras}.
\end{proposition}

\begin{definition}
\leavevmode
\begin{itemize}
	\item It is said that the algebra $\A_n \{ \gamma_0, \dots, \gamma_{n-1} \}$ is an algebra of {\em the main sequence} if $\gamma_k = -1$ for each $k = 0, \dots, n-1$. We denote this algebra by~$\Main_n$.
	\item The algebra $\A_n \{ \gamma_0, \dots, \gamma_{n-1} \}$ is called {\em a Cayley--Dickson split-algebra} if $\gamma_k = -1$ for each $k = 0, \dots, n-2$ and $\gamma_{n-1} = 1$.  We denote it by $\Hyp_n$, since the norm in $\Hyp_n$ appears to be hyperbolic.
\end{itemize}
\end{definition}

Clearly, $\Main_n$ and $\Hyp_n$ differ by the last parameter only. In other words, $\Main_n = \Main_{n-1} \{ -1 \}$ and $\Hyp_n = \Main_{n-1} \{ 1 \}$.

\begin{proposition}[{\cite[Proposition~3.31]{our_split-algebras}}] \label{proposition:A_n-euclidean-product}
\leavevmode
\begin{itemize}
	\item Let $a = \sum\limits_{j=0}^{2^n-1} a_j e^{(n)}_j, \ \ b = \sum\limits_{j=0}^{2^n-1} b_j e^{(n)}_j \in \Main_n$. Then $\langle a, b \rangle = \sum\limits_{j=0}^{2^n-1} a_j b_j$ is a Euclidean inner product.
	\item Let $a = \sum\limits_{j=0}^{2^n-1} a_j e^{(n)}_j, \ \ b = \sum\limits_{j=0}^{2^n-1} b_j e^{(n)}_j \in \Hyp_n$. Then $\langle a, b \rangle = \sum\limits_{j=0}^{2^{n-1}-1} a_j b_j - \sum\limits_{j=2^{n-1}}^{2^n-1} a_j b_j$.
\end{itemize}
\end{proposition}

\begin{example}
\leavevmode
\begin{itemize}
	\item The complex numbers $\mathbb{C}$, the quaternions $\mathbb{H}$, the octonions $\mathbb{O}$, and the sedenions $\mathbb{S}$ are the algebras of the main sequence for $n=1,\:2,\:3,$ and $4$, correspondingly, cf.~\cite{baez}.
	\item The split-complex numbers $\hat{\mathbb{C}}$, the split-quaternions $\hat{\mathbb{H}}$, the split-octonions $\hat{\mathbb{O}}$, and the split-sedenions $\hat{\mathbb{S}}$ are the split-algebras for $n=1,\:2,\:3,$ and $4$, correspondingly, see~\cite{bentz, our_split-sedenions}.
\end{itemize}
\end{example}

It follows from Proposition~\ref{proposition:A_n-euclidean-product} that in case of the algebras of the main sequence the norm is anisotropic, while in case of the split-algebras there always exist a nonzero vector of zero norm. For $n \leq 3$ any Cayley--Dickson algebra $\A_n$ is alternative, and thus it is a composition algebra. Conversely, any unital composition algebra $\A$ over a field~$\mathbb{F}$, $\chrs \mathbb{F} \neq 2$, is isomorphic to a Cayley--Dickson algebra with $n \leq 3$, see~\cite[p.~166, Theorem 2.6.2]{mccrimmon}. If such $\A$ has an anisotropic norm, then it is a division algebra, and if $\A$ has an isotropic norm, then $\A$ is a split composition algebra, cf.~\cite[p.~66]{mccrimmon}. Thus $\Main_n$, $0 \leq n \leq 3$, are division algebras, and $\Hyp_n$, $1 \leq n \leq 3$, are split composition algebras. Moreover, the following proposition shows that the only real split composition algebras are $\hat{\mathbb{C}}$, $\hat{\mathbb{H}}$, $\hat{\mathbb{O}}$. This explains the name of the split-algebras $\Hyp_n$.

\begin{proposition}[{\cite[pp.~165-166, Exercise~2.6.1A, Theorem 2.6.2]{mccrimmon}}] \label{proposition:split-algebras}
\strut\newline
It is well known that $\hat{\mathbb{H}} = \mathbb{C} \{1\} \cong \hat{\mathbb{C}} \{-1\} \cong \hat{\mathbb{C}} \{1\}$, and $\hat{\mathbb{O}} = \mathbb{H} \{1\} \cong \hat{\mathbb{H}} \{-1\} \cong \hat{\mathbb{H}} \{1\}$.
\end{proposition}

\section{Zero divisors with restrictions on alternativity}

\subsection{Octonionic subalgebras} \label{subsection:octonionic-subalgebras}

Similarly to~\cite{moreno, moreno_alternative, moreno_constructing}, we denote $\til{e}_0 = (0,e_0) \in \A_n$ and $\til{a} = a \til{e}_0$ for all $a \in \A_n$.

\begin{proposition}
Let $a = (a_1,a_2) \in \A_n$. Then $\til{a} = (\gamma_{n-1}a_2,a_1)$ and $\til{\til{a}} = \gamma_{n-1}a$.
\end{proposition}

\begin{proof}
By definition, $\til{a} = (a_1,a_2)(0,e_0) = (\gamma_{n-1}a_2,a_1)$. Hence $\til{\til{a}} = \til{(\gamma_{n-1}a_2,a_1)} = (\gamma_{n-1}a_1, \gamma_{n-1}a_2) = \gamma_{n-1}a$.
\end{proof}

\begin{corollary}
It follows from the explicit form of $\til{a}$ that for $e_j \in \A_n$ we have
$$
\til{e_j} = 
\begin{cases}
e_{j+2^{n-1}}, & 0 \leq j \leq 2^{n-1}-1,\\
\gamma_{n-1}e_{j-2^{n-1}}, & 2^{n-1} \leq j \leq 2^n-1.
\end{cases}
$$
\end{corollary}

\begin{definition}[{\cite[p. 15]{moreno_alternative}}]
The element $a \in \A_n$ {\em alternates strongly} with $b \in \A_n$ if $[a,a,b] = 0$ and $[b,b,a] = 0$.
\end{definition}

We denote $\Lin(x_1, \dots, x_k) = \mathbb{R} x_1 + \dots + \mathbb{R} x_k$. In Lemma~\ref{lemma:octonionic-subalgebra} we allow $n(a)$ and $n(b)$ to be equal to zero, in contrast to the usual definition of Cayley--Dickson algebras. We remark also that $\phi_{a,b}$ in Lemma~\ref{lemma:octonionic-subalgebra} might have nontrivial kernel, for example, if $ab = 0$ or $\til{a} = a$.

\begin{lemma}[{\cite[Lemma~5.8]{our_orthographs1}}]
\label{lemma:octonionic-subalgebra}
Let $a,b \in \A_n$ be doubly pure, $b \perp \Lin(e_0,a,\til{e}_0,\til{a})$. Let also $a$ alternate strongly with $b$. We denote $\mathbb{O}_{a,b} = \Lin(e_0,a,b,ab,\til{e}_0,\til{a},\til{b},\til{ab})$. Then there exists a surjective homomorphism $\phi_{a,b}: \A_3 \{ -n(a), -n(b), \gamma_{n-1}\} \to \mathbb{O}_{a,b}$, so $\mathbb{O}_{a,b}$ is alternative. The multiplication table in $\mathbb{O}_{a,b}$ is given by Table~\ref{table:octonionic-subalgebra}.
\end{lemma}
\begin{table}[H]
\centering
$
\begin{array}{|c|cccccccc|}
\hline
\times          & e_0             & a                     & b                     & ab                          & \til{e}_0  & \til{a}          & \til{b}          & \til{ab}               \\\hline
e_0             & e_0             & a                     & b                     & ab                          & \til{e}_0  & \til{a}          & \til{b}          & \til{ab}               \\
a               & a               & -\mu_1                & ab                    & -\mu_1 b                    & \til{a}    & -\mu_1 \til{e}_0 & -\til{ab}        & \mu_1 \til{b}          \\
b               & b               & -ab                   & -\mu_2                & \mu_2 a                     & \til{b}    & \til{ab}         & -\mu_2 \til{e}_0 & -\mu_2 \til{a}         \\
ab              & ab              & \mu_1 b               & \hphantom{.} -\mu_2 a \hphantom{.} & -\mu_1 \mu_2                & \til{ab}   & -\mu_1 \til{b}   & \mu_2 \til{a}    & -\mu_1 \mu_2 \til{e}_0 \\
\til{e}_0 & \til{e}_0 & -\til{a}        & -\til{b}        & -\til{ab}             & -\mu_3           & \mu_3 a                & \mu_3 b                & \mu_3 ab                     \\
\til{a}   & \til{a}   & \mu_1 \til{e}_0 & -\til{ab}       & \mu_1 \til{b}         & -\mu_3 a         & -\mu_1 \mu_3           & \hphantom{.} -\mu_3 ab \hphantom{.} & \mu_1 \mu_3 b                \\
\til{b}   & \til{b}   & \til{ab}        & \mu_2 \til{e}_0 & -\mu_2 \til{a}        & -\mu_3 b         & \mu_3 ab               & -\mu_2 \mu_3           & -\mu_2 \mu_3 a               \\
\hphantom{.} \til{ab} \hphantom{.} & \hphantom{.} \til{ab} \hphantom{.} & \hphantom{.} -\mu_1 \til{b} \hphantom{.} & \mu_2 \til{a}   & \hphantom{.} \mu_1 \mu_2 \til{e}_0 \hphantom{.} & \hphantom{.} -\mu_3 ab \hphantom{.} & \hphantom{.} -\mu_1 \mu_3 b \hphantom{.} & \mu_2 \mu_3 a          & \hphantom{.} -\mu_1 \mu_2 \mu_3 \hphantom{.} \\\hline
\end{array}
$
\caption{\label{table:octonionic-subalgebra} Multiplication table in $\mathbb{O}_{a,b}$.}
\end{table}

\subsection{Hexagons in orthogonality graphs}

\begin{notation} \label{notation:norm-condition}
We say that $(a,b) \in \A_{n+1}$ satisfies condition~\eqref{equation:norm-condition} if
\begin{equation*}
\begin{cases}
(n(a))^2 = (n(b))^2 \neq 0;\\
\chi := \gamma_n \dfrac{n(a)}{n(b)} = \gamma_n \dfrac{n(b)}{n(a)} = \pm 1.
\end{cases}
\tag{\textasteriskcentered} \label{equation:norm-condition}
\end{equation*}
\end{notation}

\begin{proposition} \label{proposition:self-orthogonality}
If $(a,b) \in \A_{n+1}$ is pure and satisfies condition~\eqref{equation:norm-condition} with $\chi = 1$, then $(a,b)$ is orthogonal to itself.
\end{proposition}

\begin{proof}
By definition, $n((a,b)) = n(a) - \gamma_n n(b) = \gamma_n n(b) - \gamma_n n(b) = 0$, so $(a,b)(a,b) = -(a,b)\overline{(a,b)} = -n((a,b)) = 0$.
\end{proof}

\begin{theorem}[{\cite[Remark~3.6, Corollary~3.7, Proposition~3.8]{our_orthographs1}}] \label{theorem:A_n-double-hexagon}
Let $a,b \in \A_n$ alternate strongly with $c,d \in \A_n$, $(a,b)(c,d) = 0$ in $\A_{n+1}$.
\begin{enumerate}
    \item Let $(c,d)$ satisfy condition~\eqref{equation:norm-condition}, and $(n(a))^2 + (n(b))^2 \neq 0$. Then $(a,b)$ and $(\overline{ac},-\chi da)$ also satisfy condition~\eqref{equation:norm-condition} with the same value of~$\chi$.
    \item Let $(a,b)$ and $(c,d)$ satisfy condition~\eqref{equation:norm-condition}. Then there exists the following $6$-cycle in $\Gamma_Z(\A_{n+1})$:
    $$
    (a,b) \rightarrow (c,d) \rightarrow (\overline{ac},-\chi da) \rightarrow (a,-b) \rightarrow (c,-d) \rightarrow (\overline{ac}, \chi da) \rightarrow (a,b).
    $$
    \item The equalities $(\overline{ac},-\chi da) = -\gamma_n(\bar{b}d, -\chi \gamma_n b\bar{c})$ and $(\overline{ac},\chi da) = -\gamma_n(\bar{b}d, \chi \gamma_n b\bar{c})$ hold.
\end{enumerate}
\end{theorem}

\begin{descript}
By using Theorem~\ref{theorem:A_n-double-hexagon} we obtain a subgraph of $\Gamma_Z(\A_{n+1})$ which we call a {\em hexagon}. It is depicted in Fig.~\ref{figure:directed-hexagon}.
\end{descript}

\begin{figure}[H]
\centering
\includegraphics[width=0.54\linewidth]{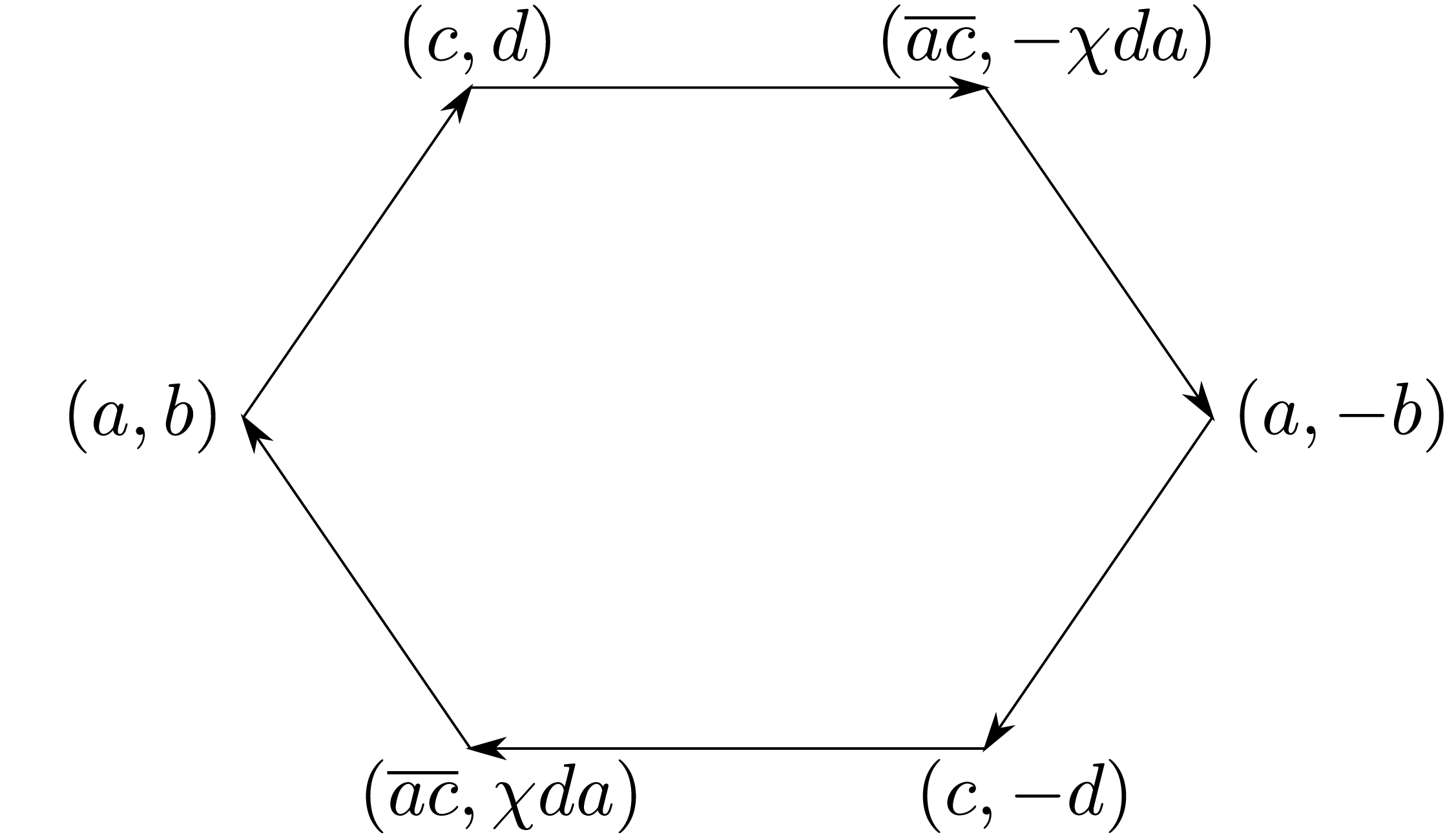}
\caption{\label{figure:directed-hexagon} A hexagon.}
\end{figure}

The following proposition establishes a relation between orthogonality graphs and zero divisor graphs of real Cayley--Dickson algebras.

\begin{proposition}[{\cite[Proposition~3.10]{our_orthographs1}}] \label{proposition:orthogonality-condition}
Let $([a], [b])$ be an edge in $\Gamma_Z(\A_n)$. Then $([a], [b])$ is also an edge in $\Gamma_O(\A_n)$ if and only if one of the following conditions holds:
\begin{enumerate}
\item $[b] = [\bar{a}]$ and $n(a) = 0$;
\item $\Re(a) = \Re(b) = 0$.
\end{enumerate}
\end{proposition}

\begin{corollary}[{\cite[Corollary~3.16]{our_orthographs1}}]
Let $a,b,c,d,ac,ad \in \A_n$ alternate strongly pairwise, $(a,b)(c,d) = 0$ in $\A_{n+1}$, $(a,b)$ and $(c,d)$ satisfy condition~\eqref{equation:norm-condition}. Let also $\Re(a) = \Re(c) = \Re(ac) = 0$. Then the hexagon in Fig.~\ref{figure:directed-hexagon} is an undirected hexagon in $\Gamma_O(\A_{n+1})$, and there are no other edges in $\Gamma_O(\A_{n+1})$ which connect its vertices, that is, there are no chords in it.
\end{corollary}

\subsection{Zero divisors of the form $\left[\left(e_i^{(n)}, \pm e_j^{(n)}\right)\right]$}

It follows from Proposition~\ref{proposition:orthogonality-condition} that any zero divisor $a \in \A_n$ with nontrivial orthogonalizer either is pure or has zero norm. If $a$ is not pure, then its connected component of $\Gamma_O(\A_n)$ consists of two vertices $[a]$ and $[\bar{a}]$. Hence, in the context of orthogonality graphs, we are interested in pure zero divisors only. Moreover, by Proposition~\ref{proposition:orthogonality-condition}, in case of pure elements, it is sufficient to verify that only one product is equal to zero.

\begin{notation}
We will need the following subsets of $Z(\A_{n+1})$:
\begin{align*}
Z_e(\A_{n+1}) &= \left\{ \left(e_i^{(n)}, \pm e_j^{(n)}\right) \in Z(\A_{n+1}) \right\} = (\E_n \times (\pm \E_n)) \cap Z(\A_{n+1});\\
Z_e'(\A_{n+1}) &= \left\{ \left(e_i^{(n)}, \pm e_j^{(n)}\right) \in Z(\A_{n+1}) \; \big| \; i \neq 0 \right\} = (\E'_n \times (\pm \E_n)) \cap Z(\A_{n+1}).
\end{align*}

Clearly, all elements of $Z_e'(\A_{n+1})$ are pure. Then $\Gamma_e(\A_{n+1})$ denotes a subgraph of $\Gamma_O(\A_{n+1})$ on the vertex set $P(Z_e'(\A_{n+1}))$.
\end{notation}

We could consider zero divisors of the form $\left(\pm e_i^{(n)}, \pm e_j^{(n)}\right)$, however, they would split into pairs of proportional elements, and each pair would correspond to the same vertex of $\Gamma_O(\A_{n+1})$. Therefore, we fix the sign of the first component, so that each line in $P(Z_e'(\A_{n+1}))$ is represented by exactly one element of $Z'_e(\A_{n+1})$. For simplicity of notation, below we define zero divisors of $\A_{n+1}$ up to multiplication by $\pm 1$.

Recall from~\cite{chan} that, in case of the sedenions, we have an elegant relation between the components of pairs of zero divisors. In other algebras, this relation takes a weaker form and holds for base units only, see~\cite[Example~4.3]{our_sedenions}.

\begin{lemma}[{\cite[Proposition~3.4(ii)]{chan}, \cite[Lemma~6.2]{our_orthographs1}}] \label{lemma:sedenions-zero-divisors-relation} \label{lemma:basis-pairs-zero-divisors-relation}
\begin{enumerate}
    \item Let $(a,b)(c,d)=0$ in $\mathbb{S}$, $n(a) = n(b) = n(c) = n(d) = 1$. Then $ab = -cd$.
    \item Let $(a,b), (c,d) \in Z_e(\A_{n+1})$, $(a,b)(c,d) = 0$. Then $ab = \pm cd$.
\end{enumerate}
\end{lemma}

\begin{notation}
Let $x \in \E_n$. Then $C(x)$ denotes the subgraph of $\Gamma_e(\A_{n+1})$ on the set of vertices $(a,b)$ such that $ab = \pm x$.
\end{notation}

It follows from Lemma~\ref{lemma:basis-pairs-zero-divisors-relation} that, given $(a,b), (c,d) \in Z_e(\A_{n+1})$ which belong to the same connected component of $\Gamma_e(\A_{n+1})$, we have $ab = \pm cd$, cf.~\cite[Corollary~6.3]{our_orthographs1}.
Hence for any $x \neq y$ the subgraphs $C(x)$ and $C(y)$ are not connected to each other in $\Gamma_e(\A_{n+1})$. Moreover, we will see in Subsection~\ref{subsection:graph-structure} that for $n \geq 3$ these subgraphs are exactly the connected components of $\Gamma_e(\A_{n+1})$, except for, possibly, $C(e_0)$.

By~\cite[Lemma~4]{schafer}, all elements of $\E_n$ are always alternative in $\A_n$.
Besides, for $(a,b) \in Z_e(\A_{n+1})$ we always have $n(a) = \pm 1$ and $n(b) = \pm 1$, and thus $(a,b)$ satisfies condition~\eqref{equation:norm-condition} for $\chi = \pm 1$, see Notation~\ref{notation:norm-condition}. The values of $\chi$ can be different for different elements of $Z_e(\A_{n+1})$, so it becomes a function of a zero divisor: $\chi = \chi((a,b))$. For simplicity, we will omit the argument of $\chi$ if it is clear of the context.

As follows from Theorem~\ref{theorem:A_n-double-hexagon}, one may prove by induction on the length of path that $\chi$ is constant on any connected component of $\Gamma_e(\A_{n+1})$. The value $\chi = \chi(C)$ is called the {\em characteristic} of the connected component $C \subseteq \Gamma_e(\A_{n+1})$. Moreover, we can determine the value of $\chi$ for every vertex of $C(x)$ explicitly, even in the case when $C(x)$ is not connected.

\begin{proposition}[{\cite[Proposition~6.5]{our_orthographs1}}] \label{proposition:chi-value}
For any $x \in \E_n$ we have $\chi(C(x)) = \gamma_n n(x)$.
\end{proposition}

Recall that $\Main_{n+1} = \Main_n \{ -1 \}$ and $\Hyp_{n+1} = \Main_n \{ 1 \}$. Since $n(x) = 1$ for any basis element $x$ of $\Main_n$, we have $\chi((a,b)) = -1$ for all $(a,b) \in Z_e(\Main_{n+1})$ and $\chi((a,b)) = 1$ for all $(a,b) \in Z_e(\Hyp_{n+1})$.

\begin{remark} \label{remark:component-types}
We consider four groups of elements $(x,y) \in Z_e'(\A_{n+1})$, depending on the type of $xy$. In this classification we assume that $e_0,a,b,c$ are linearly independent elements in $\pm \E_{n-1}$, and $ab = \pm c$.
\begin{enumerate}[label=(\Roman*)]
    \item If $xy \in \Lin(e_0)$, then $(x,y)$ is of the form $(a, \pm a)$, $(\til{a}, \pm \til{a})$, or $(\til{e}_0, \pm \til{e}_0)$.\label{item:component-type-1}
    \item If $xy \in \Lin(\til{e}_0)$, then $(x,y)$ is of the form $(a, \pm \til{a})$, $(\til{a}, \pm a)$, or $(\til{e}_0, \pm e_0)$.\label{item:component-type-2}
    \item If $xy \in \Lin(c)$, then $(x,y)$ is of the form $(c, e_0)$, $(\til{c}, \til{e}_0)$, $(\til{e}_0, \til{c})$, $(a, b)$, or $(\til{a}, \til{b})$.\label{item:component-type-3}
    \item If $xy \in \Lin(\til{c})$, then $(x,y)$ is of the form $(\til{c}, e_0)$, $(c, \til{e}_0)$, $(\til{e}_0, c)$, $(a, \til{b})$, or $(\til{a}, b)$.\label{item:component-type-4}
\end{enumerate}
Elements of distinct types (or of the same type but with linearly independent values of~$c$) belong to distinct connected components of $\Gamma_e(\A_{n+1})$.
\end{remark}

\subsection{Establishing orthogonality conditions}

The following important proposition establishes the ``almost equivalence'' of $(a,b)$, $(b, \gamma_n a)$, $(\til{a}, \til{b})$ and $(\til{b}, \gamma_n a)$ where $a,b,\til{a},\til{b}$ are linearly independent elements in $\pm \E''_n$. These elements are not properly equivalent in the sense that the sets of their neighbours do not coincide. However, all four elements have the same neighbours whose both components do not belong to $\mathbb{O}_{a,b}$, see Lemma~\ref{lemma:distinct-nonspecial-element}.

\begin{proposition}[{\cite[Proposition~6.7, Proposition~6.8]{our_orthographs1}}]
\label{proposition:shift-equivalence} \label{proposition:tilde-equivalence}
\leavevmode
\begin{enumerate}
\item Let $a,b,c,d \in \A_n$ be pure. Then $(a,b)$ is orthogonal to $(c,d)$ if and only if it is orthogonal to $(d, \gamma_n c)$.
\item Let $a,\til{a},b,\til{b},c,\til{c},d,\til{d}$ be linearly independent elements in $\pm \E''_{n}$. Then $(a,b)$ is orthogonal to $(c,d)$ if and only if it is orthogonal to $(\til{c},\til{d})$.\\
Consequently, $(a,b)$ is orthogonal to $(c,d)$ if and only if $(\til{a},\til{b})$ is orthogonal to~$(\til{c},\til{d})$.
\end{enumerate}
\end{proposition}

\begin{definition} \label{definition:special-elements}
We call an element {\em special} if one of the following conditions is satisfied:
\begin{enumerate}
    \item $\chi = 1$ and one of its components is equal to $e_0$;
    \item $\chi = -1$ and one of its components is equal to $\til{e}_0$.
\end{enumerate}
\end{definition}

We will see soon that special elements play the role of hubs in $\Gamma_e(\A_{n+1})$. Namely, for $n \geq 3$ and $x \in \E''_n$ every two elements of $C(x)$ are connected by a path of length at most three which passes through one of the special elements. Clearly, $(e_0, c) \notin Z'_e(\A_{n+1})$ for any $c \in \pm \E_n$, so we consider only elements of the form $(c, e_0)$.

\begin{notation}
We denote
\begin{align*}
    \A_{n}^{\circ} &= \A_{n} \{\gamma_0, \dots, \gamma_{n-2}, \gamma_n \},\\
    \A_{n}^{\bullet} &= \A_{n} \{\gamma_0, \dots, \gamma_{n-2}, \gamma_{n-1} \gamma_n \}.
\end{align*}
\end{notation}

\begin{lemma}[{\cite[Lemma~6.14]{our_orthographs1}}]
\label{lemma:distinct-nonspecial-element}
Let $a,b,c,d$ be linearly independent elements in $\pm \E'_{n-1}$.
\begin{enumerate}
    \item Let $x \in \{ (a, b), (b, \gamma_n a), (\til{a}, \til{b}), (\til{b}, \gamma_n \til{a}) \}$, $y \in \{ (c, d), (c, \gamma_n d), (\til{c}, \til{d}), (\til{c}, \gamma_n \til{d}) \}$. Then $x$ and $y$ are orthogonal in $\A_{n+1}$ if and only if $(a,b)$ and $(c,d)$ are orthogonal in~$\A_{n}^{\circ}$.
    \item Let $x \in \{ (a, \til{b}), (\til{b}, \gamma_n a), (\til{a}, \gamma_{n-1} b), (b, \gamma_{n-1} \gamma_n \til{a}) \}$, $y \in \{ (c, \til{d}), (\til{d}, \gamma_n c), (\til{c}, \gamma_{n-1} d),$ $(d, \gamma_{n-1} \gamma_n \til{c}) \}$. Then $x$ and $y$ are orthogonal in $\A_{n+1}$ if and only if $(a,b)$ and $(c,d)$ are orthogonal in~$\A_{n}^{\bullet}$.
\end{enumerate}
\end{lemma}

\begin{theorem}[{\cite[Lemmas~6.9,~6.10,~6.11,~6.12, and~6.15]{our_orthographs1}}] \label{theorem:orthogonality-conditions}
For any element $(x,y) \in Z'_e(\A_{n+1})$, Table~\ref{table:orthogonality-conditions} describes those elements of $Z'_e(\A_{n+1})$ which are orthogonal to it, depending on $\chi = \chi((x,y))$. We use here classification from Remark~\ref{remark:component-types}.
\end{theorem}

\begin{table}[ht!]
\centering
\scalebox{0.9}{
\begin{tabular}{|N|p{0.13\linewidth}|p{0.2\linewidth}|p{0.28\linewidth}|p{0.27\linewidth}|}
\hline
\multicolumn{1}{|c|}{\multirow{2}{*}{\textbf{Num.}}} & \multicolumn{1}{c|}{\multirow{2}{*}{\textbf{Type}}} & \multicolumn{1}{c|}{\multirow{2}{*}{\textbf{Element}}} & \multicolumn{2}{c|}{\textbf{Orthogonal elements from $Z'_e(\A_{n+1})$}}\\
\cline{4-5}
\multicolumn{1}{|c|}{} & & & \multicolumn{1}{c|}{$\chi = 1$} & \multicolumn{1}{c|}{$\chi = -1$} \\
\hline
\label{theorem:type-1-element} & Type~\ref{item:component-type-1} & $(a, \pm a)$, $a \in \E'_n$, $\chi = \gamma_n$ & Itself and all $(b, \mp b)$ such that $b \in \E'_n \setminus \{ a \}$ & Not a zero divisor \\
\hline
\label{theorem:type-2-element} & Type~\ref{item:component-type-2} without $(\til{e}_0, \pm e_0)$ & $(\til{a}, \pm a)$, $a \in \E''_n$, $\chi = - \gamma_{n-1} \gamma_n$ & Itself, $(\til{e}_0, \pm e_0)$, and $(a, \pm \gamma_n \til{a})$ & All $(\til{b}, \mp b)$ such that $b \in \E''_n \setminus \{ a, \til{a} \}$ \\
\hline
\label{theorem:e_0-special-element} & Types \ref{item:component-type-3}--\ref{item:component-type-4}, $(\til{e}_0, \pm e_0)$ & $(c, e_0)$, $c \in \pm \E'_n$ & All $(a,b) \in Z'_e(\A_{n+1})$ with $ab = n(b)c$. This case is special & Not a zero divisor \\
\hline
\label{theorem:til-e_0-special-element} & Types \ref{item:component-type-3}--\ref{item:component-type-4} & $(c, \til{e}_0)$, $(\til{e}_0, \gamma_n c)$, $c \in \pm \E''_n$, their $\chi$ are equal & Themselves, and there is a path $(\til{c}, -\gamma_{n-1} e_0) - (c, \til{e}_0) - (\til{e}_0, \gamma_n c) - (\til{c}, \gamma_{n-1} e_0)$. By Definition~\ref{definition:special-elements}, this case is nonspecial & All $(a,b) \in Z'_e(\A_{n+1})$ such that $a \perp \mathbb{H}_c = \Lin(e_0, c, \til{e}_0, \til{c})$ and $ab = \gamma_{n-1}n(b) \til{c}$. This case is special \\
\hline
\label{theorem:common-nonspecial-element} & Types \ref{item:component-type-3}--\ref{item:component-type-4} & $(a,b)$, $a,b,\til{a},\til{b} \in \pm \E''_n$ are linearly independent & Itself, $(b, \gamma_n a)$, $(c,e_0)$ with $ab = n(b)c$, and nonspecial elements from Lemma~\ref{lemma:distinct-nonspecial-element} & $(\til{a}, \til{b})$ and $(\til{b}, \gamma_n \til{a})$, $(c, \til{e}_0)$ and $(\til{e}_0, \gamma_n c)$ with $ab = \gamma_{n-1}n(b) \til{c}$, nonspecial elements from Lemma~\ref{lemma:distinct-nonspecial-element} \\
\hline
\end{tabular}
}
\caption{Orthogonality conditions in $Z'_e(\A_{n+1})$ \label{table:orthogonality-conditions}}
\end{table}

\begin{corollary}[{\cite[Theorem~6.16]{our_orthographs1}}]
\label{corollary:zero-divisors-criterion}
Let $n \geq 1$, $(a,b) \in \E'_n \times (\pm \E_n)$. Then $(a,b) \in Z(\A_{n+1})$, except for the following cases:
\begin{enumerate}
    \item $n \leq 2$ and $\chi = -1$;
    \item $b = \pm a$ and $\chi = \gamma_n = -1$;
    \item $b = \pm e_0$ and $\chi = \gamma_n n(a) = -1$.
\end{enumerate}
\end{corollary}

\section{The subgraph of $\Gamma_O'(\A_{n+1})$ on the elements $\left[\left(e_i^{(n)}, \pm e_j^{(n)}\right)\right]$} \label{section:constructing-graph}

\subsection{Low-dimensional examples} \label{subsection:low-dimensional-examples}

It follows from Proposition~\ref{proposition:split-algebras} that for $n \leq 2$ all real Cayley--Dickson algebras $\A_{n+1}$ are isomorphic either to algebras of the main sequence $\Main_{n+1}$ ($\mathbb{R}$, $\mathbb{C}$, $\mathbb{H}$, $\mathbb{O}$) or to split-algebras $\Hyp_{n+1}$ ($\hat{\mathbb{C}}$, $\hat{\mathbb{H}}$, $\hat{\mathbb{O}}$). However, if $n \leq 2$, then $\Gamma_e(\A_{n+1})$ may be nonisomorphic for isomorphic algebras. We will discuss this phenomenon more thoroughly in Subsection~\ref{subsection:distinguish-low-dimensional}. In this subsection we restrict our attention to $\Main_{n+1}$ and $\Hyp_{n+1}$ only.

It is well known that there are no zero divisors in $\Main_{n+1}$ for $n \leq 2$, so their orthogonality graphs have no vertices. $\Gamma_O(\hat{\mathbb{C}})$ consists of two vertices $[(e_0,e_0)]$ and $[(e_0,-e_0)]$, so all zero divisors are nonpure, and $V(\Gamma_e(\hat{\mathbb{C}}))$ is empty. $\Gamma_e(\hat{\mathbb{H}})$ and $\Gamma_e(\hat{\mathbb{O}})$ are depicted in Figs.~\ref{figure:H2-graph} and~\ref{figure:H3-graph}, respectively.

We now consider three particular Cayley--Dickson algebras $\A_{n+1}$ with $n \geq 3$, namely, $\mathbb{S} = \Main_4$, $\hat{\mathbb{S}} = \Hyp_4$, and $\Main_5$.

\begin{example}
By~\cite[Figs.~1 and~2]{our_sedenions}, any connected component $C(ab)$ of $\Gamma_e(\mathbb{S})$ is a double hexagon, see Fig.~\ref{figure:double-hexagon}. As follows from~\cite[Proposition~4.1]{our_sedenions}, here $e_0,a,b,c,d,ab,ac,ad$ form an orthonormal basis in $\mathbb{O}$. Besides, Lemma~\ref{lemma:sedenions-zero-divisors-relation} implies that $ab = dc = (ac)(ad)$. The whole graph $\Gamma_e(\mathbb{S})$ is depicted in Fig.~\ref{figure:M4-graph}, cf.~\cite[Fig.~2]{our_sedenions}.
\end{example}

\begin{remark} \label{remark:octonionic-notation}
We will use the same notation to describe connected components of $\Gamma_e(\hat{\mathbb{S}})$ and $\Gamma_e(\Main_5)$. Let $x \neq e_0$ be some basis element of $\mathbb{O}$. Then there exist $a,b,c,d \in \E'_3$ such that $x = ab$ and $(a,b)(c,d) = 0$ in $\mathbb{S}$. Then $\{e_0,a,b,c,d,ac,ad,x\} \subseteq \pm \E_3$ is an orthonormal basis in $\mathbb{O}$. The element $x$ can be obtained by multiplying the following pairs of elements of $\pm \E'_3$, and only them:
$$
x = ab = b \cdot (-a) = dc = c \cdot (-d) = (ac)(ad) = (ad) \cdot (-ac).
$$
\end{remark}

\begin{example}
In case of the split-sedenions, for any $(x,y) \in Z'_e(\hat{\mathbb{S}})$ we always have $n(x) = n(y) = 1$, so $\chi = \gamma_3 \frac{n(x)}{n(y)} = 1$. By~\cite[Lemma~4.21]{our_split-algebras},
$$
O_{\hat{\mathbb{S}}}((x,y)) = \left\{ \left(z, -(yz)x \right) \; \bigg| \; \Re(z) = 0, \: y(zx) = (yz)x \right\}.
$$
We are interested in $z \in \E_3$ only. If $y = \pm x$ or $y = e_0$, then we can take any $z \in \E'_n$. Otherwise, $x$ and $y$ form a quaternionic subalgebra in $\mathbb{O}$, and thus we may restrict our attention to $z \in \{ x, y, xy \}$.

The connected component $C(e_0)$ contains elements of the form $(a,\pm a)$ for $a \in \E'_3$, and $(a,a)$ is adjacent to $(b,-b)$ if and only if $a \neq b$, see Theorem~\ref{theorem:type-1-element}. We will call it {\em an almost complete $(7,7)$-bipartite graph}.

The connected component $C(ab)$ is depicted in Fig.~\ref{figure:bundle}, where $a,b,c,d$ are defined in Remark~\ref{remark:octonionic-notation}. We use here Theorem~\ref{theorem:e_0-special-element} and~\ref{theorem:common-nonspecial-element}. The graph $\Gamma_e(\hat{\mathbb{S}})$ is depicted in Fig.~\ref{figure:H4-graph}.
\end{example}

\begin{figure}[h!]
\centering
\begin{minipage}{0.5\textwidth}
\centering
\includegraphics[width=\linewidth]{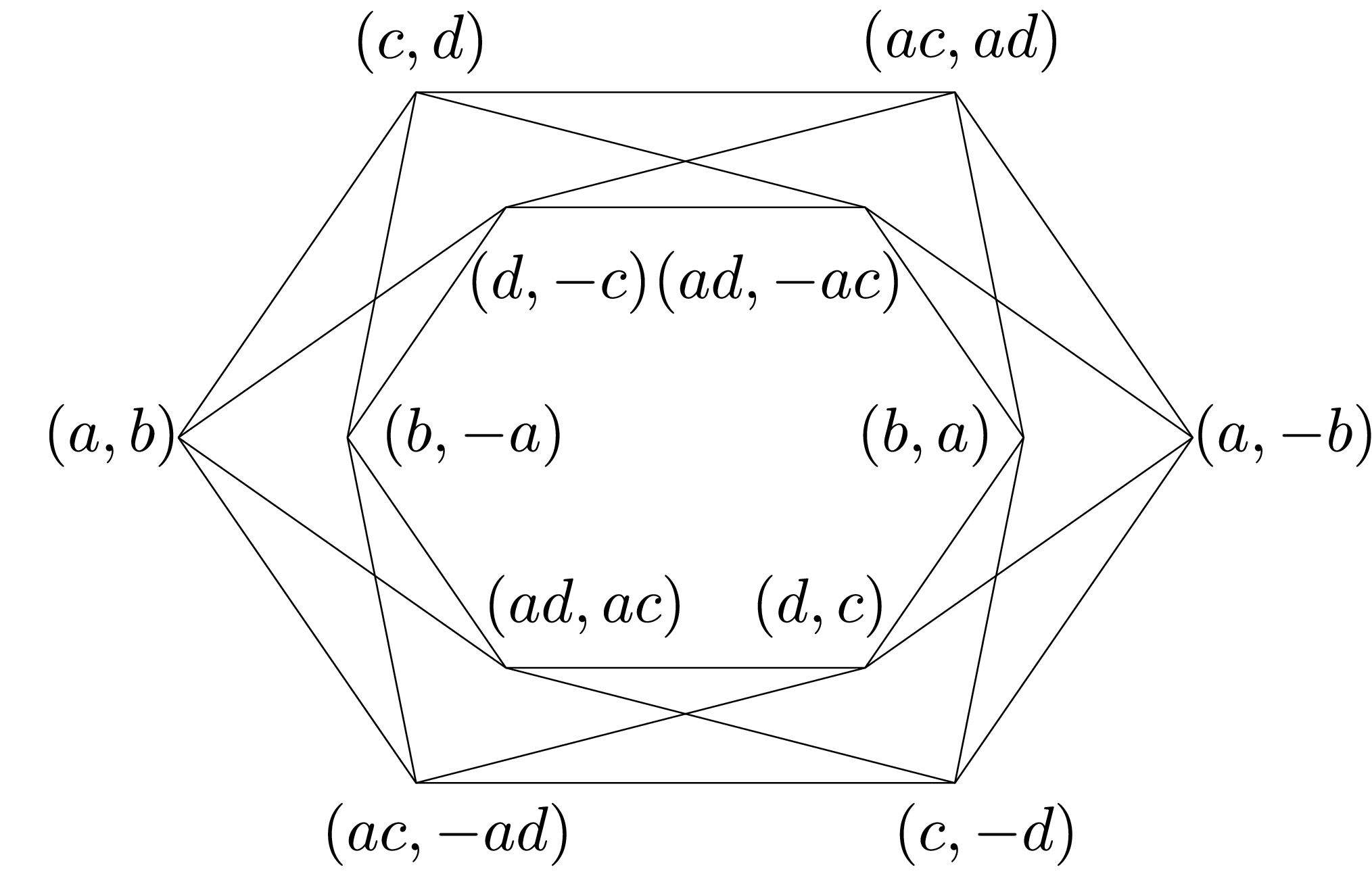}
\caption{\label{figure:double-hexagon} A double hexagon.}
\end{minipage}%
\begin{minipage}{0.5\textwidth}
\centering
\includegraphics[width=\linewidth]{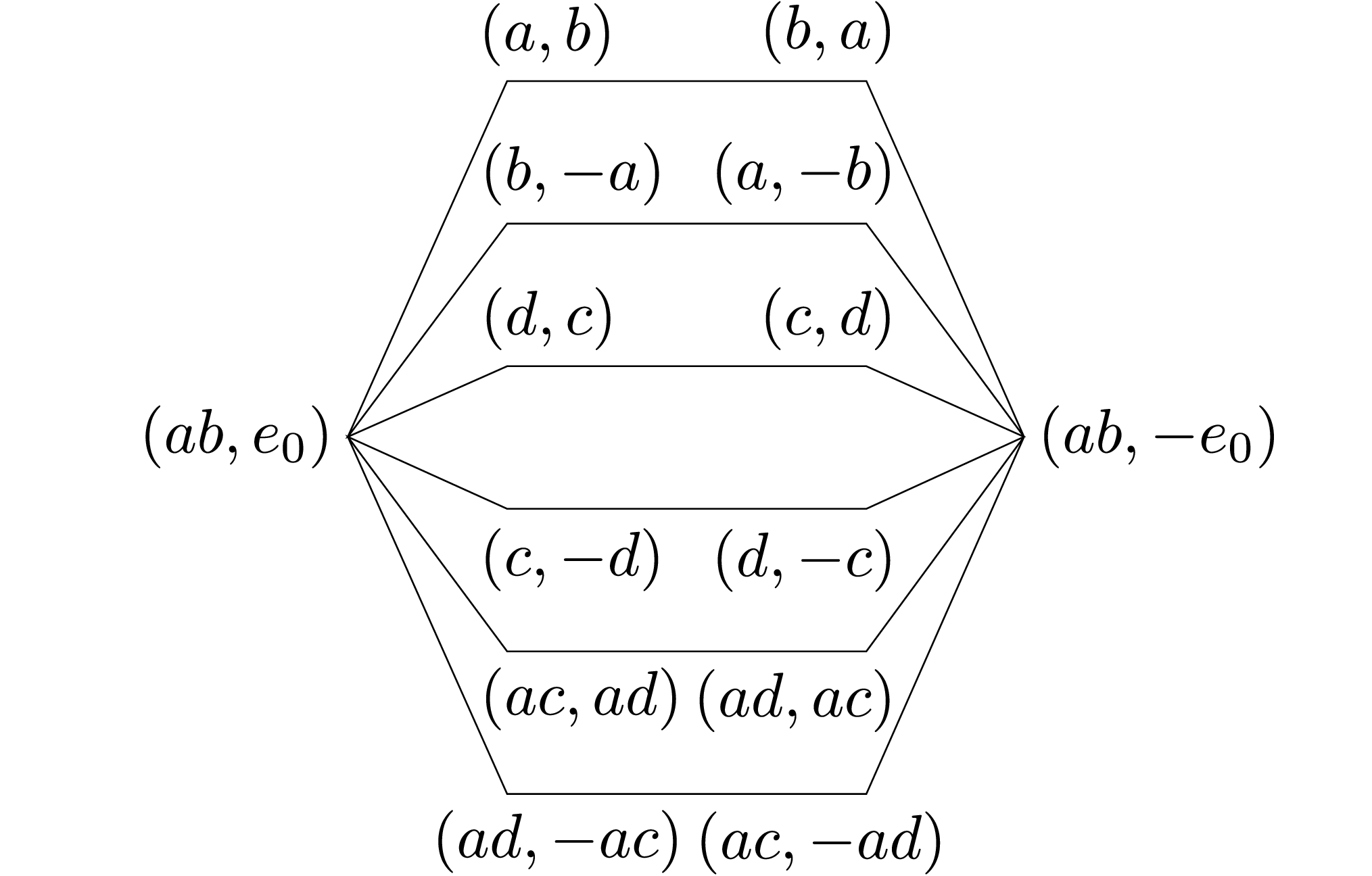}
\caption{\label{figure:bundle} A bundle of hexagons.}
\end{minipage}
\end{figure}

\begin{example}
We finally consider $\Main_5$, see Fig.~\ref{figure:M5-graph}. Note that, by~\cite[Corollary~1.12]{moreno}, any vertex $(x,y)$ in $\Gamma_O(\Main_n)$ always has the same neighbours as $(y,-x)$. Besides, $(x,y)$ and $(y,-x)$ are never adjacent. To make the picture of $\Gamma_e(\Main_5)$ more readable, from each such pair of vertices we choose the element whose components are both positive, say $(x,y)$, and replace this pair by the only vertex $(x,y)$. Thus, the double hexagon in Fig.~\ref{figure:double-hexagon} would be simplified to
$$
(a,b) \leftrightarrow (c,d) \leftrightarrow (ac,ad) \leftrightarrow (b,a) \leftrightarrow (d,c) \leftrightarrow (ad, ac) \leftrightarrow (a,b).
$$

By Theorem~\ref{theorem:type-1-element}, $\Gamma_e(\Main_5)$ has no vertices of type~\ref{item:component-type-1}, so its $C(e_0)$ is empty. The appearance of $C(\til{e}_0)$ follows from Theorem~\ref{theorem:type-2-element}, since for elements of type~\ref{item:component-type-2} we have $\chi = -\gamma_3 \gamma_4 = -(-1)(-1) = -1$. Note that for $b = \til{c}$ we have $(\til{b},\mp b) = (\til{\til{c}}, \mp \til{c}) = (\gamma_3 c, \mp \til{c}) = -(c,\pm \til{c})$.

Consider some connected component of type~\ref{item:component-type-3}, say $C(ab)$, where $a,b,c,d$ are defined in Remark~\ref{remark:octonionic-notation}. Then $\A_4^{\circ} = \A_4\{-1,-1,-1,-1\} = \mathbb{S}$, so, by Lemma~\ref{lemma:distinct-nonspecial-element}, multiplication in $\Main_5$ is reduced to multiplication in $\mathbb{S}$. It remains to use Fig.~\ref{figure:double-hexagon} for middle hexagons and Theorem~\ref{theorem:til-e_0-special-element} and~\ref{theorem:common-nonspecial-element} for additional edges which serve two purposes. Firstly, they connect all elements either to $(\til{e}_0,\til{ab})$ or to $(\til{ab},\til{e}_0)$. Secondly, for $x,y \in \E'_3$ they connect $(x,y)$ and $(\til{x},\til{y})$. When passing by a solid edge, the product of the components changes sign, as it occurs for the sedenions, cf. Lemma~\ref{lemma:sedenions-zero-divisors-relation}. When passing by a dashed edge, the product of the components is preserved. According to Lemma~\ref{lemma:basis-pairs-zero-divisors-relation}, these are the only two opportunities.

Let $C(\til{ab})$ be a connected component of type~\ref{item:component-type-4}. Since $\A_4^{\bullet} = \A_4\{-1,-1,-1,
\\
(-1)(-1)\} = \hat{\mathbb{S}}$, by Lemma~\ref{lemma:distinct-nonspecial-element}, multiplication in $\Main_5$ is reduced to multiplication in $\hat{\mathbb{S}}$. However, $\Gamma_e(\hat{\mathbb{S}})$ contains no doubly pure elements with linearly independent components which are adjacent. Hence we use Theorem~\ref{theorem:til-e_0-special-element} and~\ref{theorem:common-nonspecial-element} to obtain $C(\til{ab})$. Note that $(\til{a},b)$ is orthogonal to $(\gamma_4 \til{b},\til{\til{a}}) = (\gamma_4 \til{b}, \gamma_3 a) = -(\til{b},a)$.
\end{example}

\begin{figure}[h!]
\centering
\includegraphics[width=\linewidth]{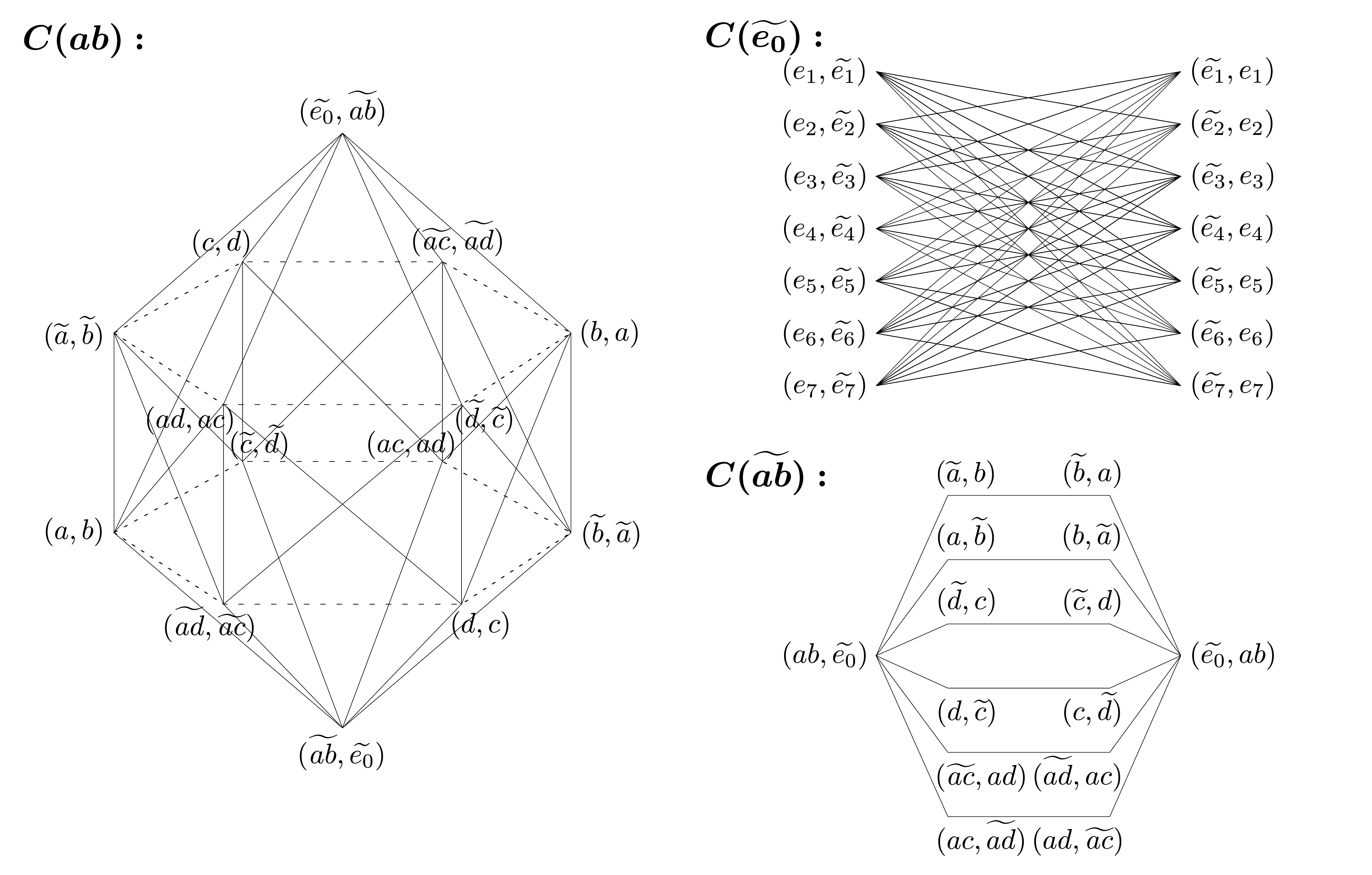}
\caption{\label{figure:M5-graph} Connected components of $\Gamma_e(\Main_5)$.}
\end{figure}

\subsection{Graph structure} \label{subsection:graph-structure}

We have described some graphs $\Gamma_e(\Main_{n+1})$ and $\Gamma_e(\Hyp_{n+1})$ in Subsection~\ref{subsection:low-dimensional-examples}. We now generalize these results to the connected components of $\Gamma_e(\A_{n+1})$ in general and find their diameters.

We consider first elements of types~\ref{item:component-type-1} and~\ref{item:component-type-2} from Remark~\ref{remark:component-types}. We denote $N = 2^{n-1} - 1$.

\begin{lemma} \label{lemma:special-components}
$C(e_0)$ and $C(\til{e}_0)$ are depicted in Fig.~\ref{figure:special-components} for various values of $\gamma_n$ and $\gamma_{n-1}\gamma_n$. Note that, if $\gamma_n = -1$, then $C(e_0)$ has an empty set of vertices. For $n \geq 3$ the subgraphs $C(e_0)$ and $C(\til{e}_0)$, if nonempty, are connected components of diameter three.
\end{lemma}

\begin{proof}
The appearance of $C(e_0)$ follows immediately from Theorem~\ref{theorem:type-1-element}.

For $C(\til{e}_0)$ we apply Theorem~\ref{theorem:type-2-element} and~\ref{theorem:e_0-special-element}. The case when $\gamma_{n-1}\gamma_n = -1$ is straightforward. In the case when $\gamma_{n-1}\gamma_n = 1$ it remains to note that for $b = \til{c}$ we have $(\til{b},\mp b) = (\til{\til{c}}, \mp \til{c}) = (\gamma_{n-1} c, \mp \til{c}) = \gamma_n (c,\mp \gamma_n \til{c})$.

The connectedness and the diameter value are verified directly.
\end{proof}

\begin{figure}[h!]
\centering
\includegraphics[width=\linewidth]{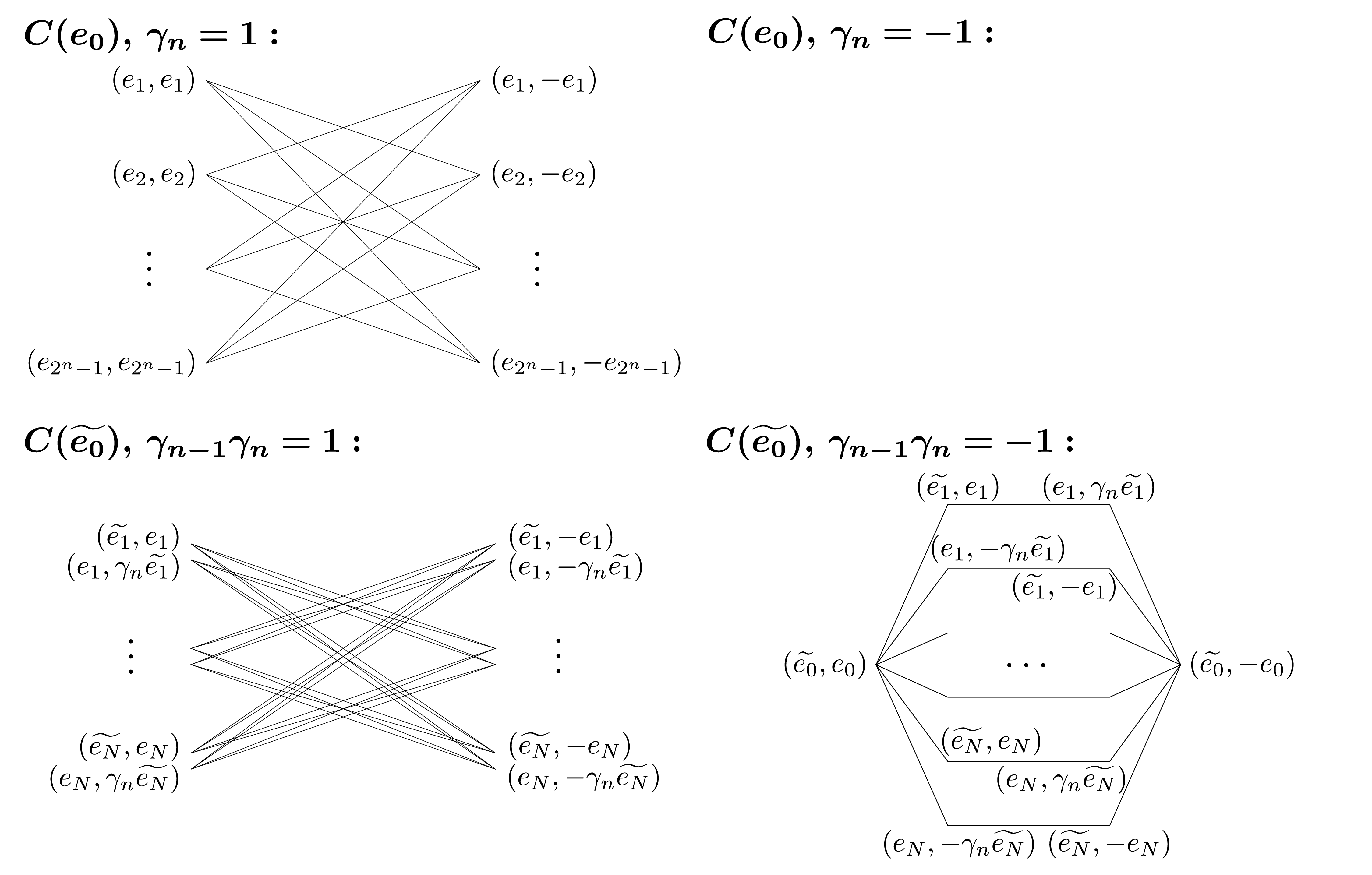}
\caption{\label{figure:special-components} $C(e_0)$ and $C(\til{e}_0)$ for various values of $\gamma_{n-1}$ and $\gamma_n$.}
\end{figure}

We now show that for $n \geq 3$ and any $x \in \E'_{n-1}$ the subgraphs $C(x)$ are $C(\til{x})$ are also connected components of $\Gamma_e(\A_{n+1})$ with diameter three. Recall that special elements are given by Definition~\ref{definition:special-elements}.

\begin{theorem} \label{theorem:nonspecial-components}
Let $x \in \E'_{n-1}$. The structure of the subgraphs $C(x)$ are $C(\til{x})$ is depicted in Fig.~\ref{figure:nonspecial-components} for various values of $\chi$. Here $a,b,c,d,x \in \E'_{n-1}$ are linearly independent, $ab = n(b)x$ and $cd = \pm n(c)x$. A dashed edge connects two vertices in $C(ab)$ if and only if $(a,b)$ and $(c,d)$ are orthogonal in $\A_n^{\circ}$, while a dotted edge connects two vertices in $C(\til{ab})$ if and only if $(a,b)$ and $(c,d)$ are orthogonal in $\A_n^{\bullet}$, see Lemma~\ref{lemma:distinct-nonspecial-element}.
\end{theorem}

\begin{figure}[h!]
\centering
\includegraphics[width=\linewidth]{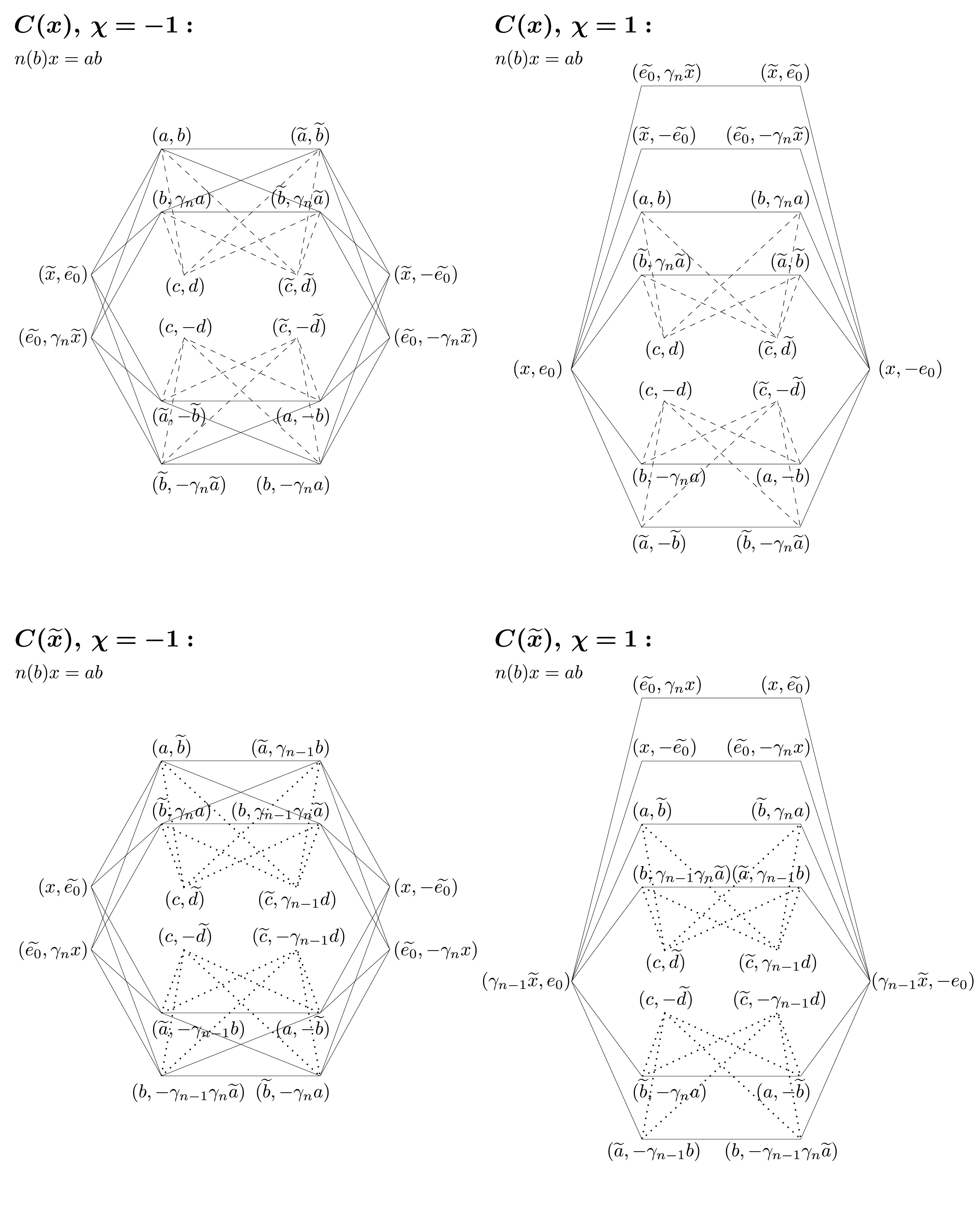}
\caption{\label{figure:nonspecial-components} Connected components of types~\ref{item:component-type-3} and~\ref{item:component-type-4} for various values of $\chi$.}
\end{figure}

\begin{proof}
We consider all cases independently. Note that if $a,b \in \pm \E'_{n-1}$ are such that $(a,b) \in C(x)$ or $(a,\til{b}) \in C(\til{x})$, then $ab = \pm x \neq \pm e_0$. Hence $e_0,a,b$ are linearly independent, and thus $a$ and $b$ satisfy the conditions of Lemma~\ref{lemma:octonionic-subalgebra}, that is, $\til{e}_0,a,b$ form an octonionic subalgebra $\mathbb{O}_{a,b}$.
\begin{enumerate}
    \item Let $\chi = -1$ for $C(x)$. By Theorem~\ref{theorem:til-e_0-special-element}, $(\til{x},\til{e}_0)$ and $(\til{e}_0, \gamma_n \til{x})$ are orthogonal to those $(a,b) \in Z'_e(\A_{n+1})$ which satisfy $a \perp \Lin(e_0, x, \til{e}_0, \til{x})$ and $ab = \gamma_{n-1}n(b)\til{\til{x}} = (\gamma_{n-1})^2 n(b)x = n(b)x$. We now assume that $a,b \in \pm \E'_{n-1}$ and $ab = n(b)x$, so $(\til{x},\til{e}_0)$ and $(a,b)$ are indeed orthogonal. We apply Theorem~\ref{theorem:A_n-double-hexagon} to $(\til{x},\til{e}_0)$ and $(a,b)$ and obtain immediately the following hexagon:
    $$
    (\til{x},\til{e}_0) \leftrightarrow (a,b) \leftrightarrow (\til{a},\til{b}) \leftrightarrow (\til{x},-\til{e}_0) \leftrightarrow (a,-b) \leftrightarrow (\til{a},-\til{b}) \leftrightarrow (\til{x},\til{e}_0).
    $$
    Since $\chi = -1$ and all our elements are doubly pure, by Proposition~\ref{proposition:shift-equivalence}, we can say that any $(A,B)$ has the same neighbours as $(B,\gamma_n A)$. Hence we obtain a double hexagon consisting of $(\til{x},\pm \til{e}_0)$, $(\til{e}_0,\pm \gamma_n\til{x})$, $(a, \pm b)$, $(\til{a}, \pm \til{b})$, $(b, \pm \gamma_n a)$, and $(\til{b}, \pm \til{a})$.
    
    Observe that we could construct this double hexagon by using Theorem~\ref{theorem:til-e_0-special-element} and~\ref{theorem:common-nonspecial-element}. Namely, we could find explicitly which pairs of its vertices are adjacent, and which are not.
    
    Let now $a,b,c,d \in \E'_{n-1}$ be linearly independent. We apply Lemma~\ref{lemma:distinct-nonspecial-element} which states that any element of $\{ (a, b), (b, \gamma_n a), (\til{a}, \til{b}), (\til{b}, \gamma_n \til{a}) \}$ is orthogonal to any element of $\{ (c, d), (c, \gamma_n d), (\til{c}, \til{d}), (\til{c}, \gamma_n \til{d}) \}$ if and only if $(a,b)$ and $(c,d)$ are orthogonal in~$\A_{n}^{\circ}$. We depict these edges by dashed lines.
    
    Note that $cd = \pm ab$, and thus $cd = \pm n(c)x$. Hence $(c,d)$ is orthogonal either to $(\til{x},\til{e}_0)$ or to $(\til{x},-\til{e}_0)$. Then we may obtain a double hexagon similar to that which contains $(a,b)$, where $(c,d)$ plays the role either of $(a,b)$ or of $(a,-b)$.
    
    \item Let $\chi = 1$ for $C(x)$. By Theorem~\ref{theorem:e_0-special-element}, $(x,e_0)$ is orthogonal to those $(a,b) \in Z'_e(\A_{n+1})$ which satisfy $ab = n(b)x$. Let now $a,b \in \pm \E'_{n-1}$ and $ab = n(b)x$. We apply Theorem~\ref{theorem:A_n-double-hexagon} to $(x, e_0)$ and $(a,b)$ and obtain the following hexagon:
    $$
    (x,e_0) \leftrightarrow (a,b) \leftrightarrow (b,\gamma_n a) \leftrightarrow (x,-e_0) \leftrightarrow (a,-b) \leftrightarrow (b,-\gamma_n a) \leftrightarrow (x,e_0).
    $$
    Note also that, by Lemma~\ref{lemma:octonionic-subalgebra},
    $$
    \til{a}\til{b} = -n(\til{e}_0)ab = -n(\til{e}_0)n(b)x = -n(\til{b})x,
    $$
    so $(\til{a},\til{b})$ is orthogonal to $(x, -e_0)$. Similarly, we obtain the hexagon
    $$
    (x,-e_0) \leftrightarrow (\til{a},\til{b}) \leftrightarrow (\til{b},\gamma_n \til{a}) \leftrightarrow (x,e_0) \leftrightarrow (\til{a},-\til{b}) \leftrightarrow (\til{b},-\gamma_n \til{a}) \leftrightarrow (x,-e_0).
    $$
    It follows from Theorem~\ref{theorem:common-nonspecial-element} that there are no other edges between the vertices of these two hexagons.
    
    Let $a,b,c,d \in \E'_{n-1}$ be linearly independent. By Lemma~\ref{lemma:distinct-nonspecial-element}, any element of $\{ (a, b), (b, \gamma_n a), (\til{a}, \til{b}), (\til{b}, \gamma_n \til{a}) \}$ is orthogonal to any element of $\{ (c, d), (c, \gamma_n d), (\til{c}, \til{d}),$ $(\til{c}, \gamma_n \til{d}) \}$ if and only if $(a,b)$ and $(c,d)$ are orthogonal in~$\A_{n}^{\circ}$.
    
    Since $cd = \pm n(c)x$, the element $(c,d)$ is orthogonal either to $(x,e_0)$ or to $(x,-e_0)$. Hence we may obtain two hexagons similar to those which contain $(a,b)$.
    
    We use Theorem~\ref{theorem:til-e_0-special-element} to obtain the hexagon which contains $(\til{x},\til{e}_0)$:
    \begin{align*}
    (x,-e_0) &\leftrightarrow (\til{x},\til{e}_0) \leftrightarrow (\til{e}_0,\gamma_n \til{x}) \leftrightarrow (x,e_0)\\
    &\leftrightarrow (\til{x},-\til{e}_0) \leftrightarrow (\til{e}_0,-\gamma_n \til{x}) \leftrightarrow (x,-e_0).
    \end{align*}
    The elements $(\til{x},\pm \til{e}_0)$ and $(\til{e}_0,\pm \gamma_n \til{x})$ are connected to no other vertices of $C(x)$.
    
    \item Let $\chi = -1$ for $C(\til{x})$. Assume that $a,b \in \pm \E'_{n-1}$ are such that $(x,\til{e}_0)$ and $(a,\til{b})$ are orthogonal. By Theorem~\ref{theorem:til-e_0-special-element}, this is equivalent to the fact that $a \perp \Lin(e_0, x, \til{e}_0, \til{x})$ and $-\til{ab} = a\til{b} = \gamma_{n-1}n(\til{b})\til{x} = \gamma_{n-1}n(b)n(\til{e}_0)\til{x} = -(\gamma_{n-1})^2 n(b)\til{x} = -n(b)\til{x}$, that is, $ab = n(b)x$. We apply Theorem~\ref{theorem:A_n-double-hexagon} to $(x,\til{e}_0)$ and $(a,\til{b})$ and obtain the following hexagon:
    $$
    (x,\til{e}_0) \leftrightarrow (a,\til{b}) \leftrightarrow (\til{a},\gamma_{n-1}b) \leftrightarrow (x,-\til{e}_0) \leftrightarrow (a,-\til{b}) \leftrightarrow (\til{a},-\gamma_{n-1}b) \leftrightarrow (x,\til{e}_0).
    $$
    
    The rest of the proof is similar to the case of $C(x)$ for $\chi = -1$, with the only exception. Let $a,b,c,d \in \E'_{n-1}$ be linearly independent. By Lemma~\ref{lemma:distinct-nonspecial-element}, any element of $\{ (a, \til{b}), (\til{b}, \gamma_n a), (\til{a}, \gamma_{n-1} b), (b, \gamma_{n-1} \gamma_n \til{a}) \}$ is orthogonal to any element of $\{ (c, \til{d}), (\til{d}, \gamma_n c), (\til{c}, \gamma_{n-1} d), (d, \gamma_{n-1} \gamma_n \til{c}) \}$ if and only if $(a,b)$ and $(c,d)$ are orthogonal in~$\A_{n}^{\bullet}$. We depict these edges by dotted lines.
    
    \item Let $\chi = 1$ for $C(\til{x})$. Assume that $a,b \in \pm \E'_{n-1}$ are such that $(\gamma_{n-1}\til{x},e_0)$ and $(a,\til{b})$ are orthogonal. By Theorem~\ref{theorem:e_0-special-element}, this is equivalent to the fact that $-\til{ab} = a\til{b} = n(\til{b})(\gamma_{n-1}\til{x}) = -n(b)\til{x}$, that is, $ab = n(b)x$. We apply Theorem~\ref{theorem:A_n-double-hexagon} to $(\gamma_{n-1}\til{x},e_0)$ and $(a,\til{b})$ and obtain the following hexagon:
    \begin{multline*}
    (\gamma_{n-1}\til{x},e_0) \leftrightarrow (a,\til{b}) \leftrightarrow (\til{b},\gamma_n a) \leftrightarrow (\gamma_{n-1}\til{x},-e_0) \leftrightarrow\\
    \leftrightarrow (a,-\til{b}) \leftrightarrow (\til{b},-\gamma_n a) \leftrightarrow (\gamma_{n-1}\til{x},e_0).
    \end{multline*}
    Note also that, by Lemma~\ref{lemma:octonionic-subalgebra},
    $$
    \til{a}(\gamma_{n-1}b) = -\gamma_{n-1} \til{ab} = n(b)(-\gamma_{n-1} \til{x}).
    $$
    Then, by Theorem~\ref{theorem:e_0-special-element}, $(\til{a},\gamma_{n-1}b)$ is orthogonal to $(\gamma_{n-1} \til{x}, -e_0)$. Similarly, we obtain the hexagon
    \begin{multline*}
    (\gamma_{n-1}\til{x},-e_0) \leftrightarrow (\til{a},\gamma_{n-1}b) \leftrightarrow (b,\gamma_{n-1}\gamma_n \til{a}) \leftrightarrow (\gamma_{n-1}\til{x},e_0) \leftrightarrow\\
    \leftrightarrow (\til{a},-\gamma_{n-1}b) \leftrightarrow (b,-\gamma_{n-1}\gamma_n \til{a}) \leftrightarrow (\gamma_{n-1}\til{x},-e_0).
    \end{multline*}
    
    We use Theorem~\ref{theorem:til-e_0-special-element} to obtain the hexagon which contains $(x,\til{e}_0)$:
    \begin{multline*}
    (\gamma_{n-1}\til{x},-e_0) \leftrightarrow (x,\til{e}_0) \leftrightarrow (\til{e}_0,\gamma_n x) \leftrightarrow (\gamma_{n-1}\til{x},e_0) \leftrightarrow\\
    \leftrightarrow (x,-\til{e}_0) \leftrightarrow (\til{e}_0,-\gamma_n x) \leftrightarrow (\gamma_{n-1}\til{x},-e_0).
    \end{multline*}
    
    Let now $a,b,c,d \in \E'_{n-1}$ be linearly independent. Then, by Lemma~\ref{lemma:distinct-nonspecial-element}, any element of $\{ (a, \til{b}), (\til{b}, \gamma_n a), (\til{a}, \gamma_{n-1} b), (b, \gamma_{n-1} \gamma_n \til{a}) \}$ is orthogonal to any element of $\{ (c, \til{d}), (\til{d}, \gamma_n c), (\til{c}, \gamma_{n-1} d), (d, \gamma_{n-1} \gamma_n \til{c}) \}$ if and only if $(a,b)$ and $(c,d)$ are orthogonal in~$\A_{n}^{\bullet}$. \qedhere
\end{enumerate}
\end{proof}

\begin{remark}
The elements $(d,\pm \gamma_n c)$ and $(\til{d}, \pm \gamma_n \til{c})$, and also $(\til{d},\pm \gamma_n c)$ and $(d, \pm \gamma_{n-1} \gamma_n \til{c})$ are not depicted in $C(x)$ and $C(\til{x})$, respectively. However, one may determine their nonspecial neighbours from Lemma~\ref{lemma:distinct-nonspecial-element}.

We obtain by analogy with $(a,\pm b)$ and $(\til{a}, \pm \til{b})$ that, in case of $C(x)$, $(c,d)$ and $(\til{c},-\til{d})$ are connected to special elements on one side, while $(c,-d)$ and $(\til{c},\til{d})$ are connected to special elements on the other side. In case of $C(\til{x})$, $(c,\til{d})$ and $(\til{c},-\gamma_{n-1} d)$ are connected to special elements on one side, while $(c,-\til{d})$ and $(\til{c},\gamma_{n-1} d)$ are connected to special elements on the other side, again by analogy with $(a,\pm \til{b})$ and $(\til{a}, \pm \gamma_{n-1}b)$.
\end{remark}

\begin{corollary} \label{corollary:diameters}
If $n \geq 3$, then the diameters of $C(x)$ and $C(\til{x})$ are equal to three.
\end{corollary}

\begin{proof}
It follows from Theorem~\ref{theorem:nonspecial-components} that every nonspecial element $A$ belongs to a hexagon (or a double hexagon) whose two opposite vertices (or pairs of vertices) are special. For $\chi = 1$ we have a hexagon, and for $\chi = -1$ we have a double hexagon. Hence $A$ lies at distance one from one special element (or a pair of elements), and at distance two from the other special element (or a pair of elements).

To connect nonspecial elements $A$ and $B$, we first choose the nearest special element to~$A$ (at distance one), and then proceed to $B$ (at distance at most two), hence $d(A,B) \leq 3$. Since $n \geq 3$, Corollary~\ref{corollary:zero-divisors-criterion} implies that nonspecial zero divisors do exist, so we may connect any two special elements by a path of length at most three. It remains to note that the distance between special elements on the opposite sides is exactly three.
\end{proof}

\begin{corollary} \label{corollary:retrieve-n-and-chi}
If $n \geq 3$, then $\Gamma_e(\A_{n+1})$ contains either $2^n$ or $2^n - 1$ connected components for $\gamma_n = 1$ and $\gamma_n = -1$, respectively. Let $C$ be an arbitrary connected component of $\Gamma_e(\A_{n+1})$. Then the diameter of $C$ equals three. Moreover, we can find the number of vertices in $C$ as follows:
$$
|V(C)| =
\begin{cases}
2^{n+1}-2, & \chi(C) = 1;\\
2^{n+1}-4, & \chi(C) = -1.
\end{cases}
$$
\end{corollary}

\begin{proof}
We have shown in Lemma~\ref{lemma:special-components} and Theorem~\ref{theorem:nonspecial-components} that $C(x)$ is a connected component of $\Gamma_e(\A_{n+1})$ for any $x \in \E'_n$. The subgraph $C(e_0)$ is nonempty if and only if $\gamma_n = 1$, and it is also a connected component in this case. Thus we obtain the number of the connected components of $\Gamma_e(\A_{n+1})$. The values of their diameters follow from Lemma~\ref{lemma:special-components} and Corollary~\ref{corollary:diameters}.

We use Corollary~\ref{corollary:zero-divisors-criterion} to compute the number of vertices in $C(x)$, $x \in \E_n$. We are interested in such $(a,b) \in Z_e(\A_{n+1})$ that $ab = \pm x$. If $\chi = 1$, then we may take any $a \in \E'_n$ and find two corresponding values of $b$ for both $x$ and $-x$. Hence we have $2(2^n-1) = 2^{n+1}-2$ elements. If $\chi = -1$, then we can only take $a \in \E'_n \setminus \{ x \}$. Since $x \neq e_0$ in this case, there are $2(2^n-2) = 2^{n+1}-4$ choices of $(a,b)$.
\end{proof}

\begin{proposition}
For any $n \geq 1$ all connected components of $\Gamma_e(\A_{n+1})$ are Eulerian graphs.
\end{proposition}

\begin{proof}
It can be easily seen from Figs.~\ref{figure:special-components} and~\ref{figure:nonspecial-components} that the degree of every vertex of $C(x)$ is even. In particular, we use here Lemma~\ref{lemma:distinct-nonspecial-element} which states that all nonspecial elements of types~\ref{item:component-type-3} and~\ref{item:component-type-4} are divided into groups of ``almost equivalent'' elements, four of them in each. This implies that the dotted or the dashed degree of every nonspecial vertex is divided by $4$.

It is also possible that $C(x)$ has an empty set of vertices or consists of two disconnected vertices, see Theorem~\ref{theorem:low-dimensional-graphs} further. However, these graphs are Eulerian as well.
\end{proof}

\section{Retrieving parameters from $\Gamma_e(\A_{n+1})$} \label{section:retrieving-parameters}

\subsection{Finding $\gamma_{n-1}$ and $\gamma_n$ for $n \geq 4$} \label{subsection:finding-last-two}

For now, our goal is to retrieve all values $\gamma_k$ of $\A_{n+1}$, $k = 0, \dots, n$, from the graph $\Gamma_e(\A_{n+1})$. For $n \geq 3$ Corollary~\ref{corollary:retrieve-n-and-chi} provides a method to determine the value of $n$ from the number of the connected components of $\Gamma_e(\A_{n+1})$. Then we can count the elements of an arbitrary connected component and find its $\chi$.

We first learn how to find special elements from Definition~\ref{definition:special-elements} in connected components of the form $C(x)$ with $x \in \E''_n$.

\begin{lemma} \label{lemma:find-special-elements}
Let $n \geq 2$ and $x \in \E''_n$.
\begin{enumerate}
    \item If $\chi = 1$, then $C(x)$ has a nonempty set of vertices.
    If $n = 2$, then $C(x)$ is a hexagon, so we can take any two opposite elements as special. If $n \geq 3$, then the special elements of $C(x)$ are uniquely defined.
    \item If $\chi = -1$, then $V(C(x))$ is nonempty if and only if $n \geq 3$. If $n = 3$, then $C(x)$ is simply a double hexagon, so we can take any two opposite pairs of elements as special. If $n \geq 4$, then the special elements of $C(x)$ are uniquely defined.
\end{enumerate}
\end{lemma}

\begin{proof}
Clearly, $\E''_n \neq \varnothing$ if and only if $n \geq 2$, so this condition is essential.
\begin{enumerate}
    \item Let $\chi = 1$. By Corollary~\ref{corollary:zero-divisors-criterion}, any $(a,b) \in \E'_n \times (\pm \E_n)$ with $ab = x$ is a zero divisor, and such $(a,b)$ do exist, so $V(C(x))$ is nonempty.
    
    If $n = 2$, then $\E_n = \{ e_0, x, \til{e}_0, \til{x} \}$, up to signs. Hence $C(x)$ is just a hexagon with no other edges:
    $$
    (x,-e_0) \leftrightarrow (\til{x},\til{e}_0) \leftrightarrow (\til{e}_0,\gamma_n \til{x}) \leftrightarrow (x,e_0) \leftrightarrow (\til{x},-\til{e}_0) \leftrightarrow (\til{e}_0,-\gamma_n \til{x}) \leftrightarrow (x,-e_0).
    $$
    Clearly, any two opposite elements can be regarded as special up to a graph isomorphism.
    
    If $n \geq 3$, then there exist $a,b \in \E''_n$ such that $ab = n(b)x$, and we obtain at least two more hexagons passing through $(x,e_0)$ and $(x,-e_0)$. The distance from $(x, \pm e_0)$ to any nonspecial element is at most two, so the only element at a distance three from a special element is the other special element. However, the distance from $(\til{x},\til{e}_0)$ and $(\til{e}_0,-\gamma_n{\til{x}})$ to $(a,b)$ and $(\til{a},-\til{b})$ is exactly three, and the distance from $(\til{x},-\til{e}_0)$ and $(\til{e}_0,\gamma_n{\til{x}})$ to $(a,-b)$ and $(\til{a},\til{b})$ is also three. Hence none of them can be mapped to a special element by a graph isomorphism, and thus $(x,e_0)$ and $(x,-e_0)$ are uniquely defined.
    
    \item Let $\chi = -1$. By Corollary~\ref{corollary:zero-divisors-criterion}, $V(C(x))$ is nonempty only for $n \geq 3$.
    
    If $n = 3$, then $\E_n = \{ e_0, x, \til{e}_0, \til{x}, a, b, \til{a}, \til{b} \}$, up to signs, for some $a,b$ with $ab = n(b)x$. Therefore, $C(x)$ is a double hexagon consisting of $(\til{x},\pm \til{e}_0)$, $(\til{e}_0,\pm \gamma_n\til{x})$, $(a, \pm b)$, $(\til{a}, \pm \til{b})$, $(b, \pm \gamma_n a)$, and $(\til{b}, \pm \til{a})$.    Any two opposite pairs of elements can be regarded as special up to a graph isomorphism.
    
    If $n \geq 4$, then we use the degrees of special and nonspecial elements. By Corollary~\ref{corollary:retrieve-n-and-chi}, $C(x)$ consists of $2^{n+1} - 4$ vertices. Exactly $2^{n+1} - 8$ of them are nonspecial. Every special element is connected to half of them, that is, its degree equals $2^n - 4 \geq 8$. Assume now that some neighbour, say $(a,b)$, of a special element, say $(\til{x},\til{e}_0)$, has degree $K$. Then $K - 4$ of its neighbours do not belong to the double hexagon which contains $(a,b)$ and $(\til{x},\til{e}_0)$. Half of them, that is, $\frac{K - 4}{2}$ elements, are connected to $(\til{x},\til{e}_0)$, and another half are connected to $(\til{x},-\til{e}_0)$.
    
    Assume now that some nonspecial element $(a,b)$ is mapped to a special element by a graph isomorphism. Let $K$ denote the degree of $(a,b)$. Then $K = 2^n - 4 \geq 8$. The degree of its neighbour $(\til{a},\til{b})$ is also $K$. By Lemma~\ref{lemma:distinct-nonspecial-element}, every neighbour of $(a,b)$ outside the double hexagon is also a neighbour of $(\til{a},\til{b})$, and vise versa. Hence they have $K - 4 > \frac{K - 4}{2}$ common neighbours, a contradiction. Therefore, special elements are uniquely defined. \qedhere
\end{enumerate}
\end{proof}

\begin{lemma} \label{lemma:find-special-components}
If $n \geq 3$, then we can determine $\gamma_n$ and, if $\gamma_n = 1$, find $C(e_0)$.

If $n \geq 4$, then we can also find $C(\til{e}_0)$ and thus determine $\gamma_{n-1}$.
\end{lemma}

\begin{proof}
Let $n \geq 3$. By Corollary~\ref{corollary:retrieve-n-and-chi}, the number of connected components equals either $2^n$ or $2^n - 1$ for $\gamma_n = 1$ and $\gamma_n = -1$, respectively. Hence we can count them and find $\gamma_n$. If $\gamma_n = -1$, then $C(e_0)$ has an empty set of vertices. If $\gamma_n = 1$, then $C(e_0)$ is uniquely described by Lemma~\ref{lemma:special-components}, so we can simply find a component which is isomorphic to $C(e_0)$ and regard it as $C(e_0)$. We further consider $\Gamma_e(\A_{n+1})$ without $C(e_0)$.

Let now $n \geq 4$, and we want to find $C(\til{e}_0)$. It is depicted in Fig.~\ref{figure:special-components} for various values of~$\chi$. If its $\chi = -\gamma_{n-1}\gamma_n = -1$, then the degree of any element equals $2(N-1) = 2^n - 4 \geq 8$. However, any two adjacent vertices have no neighbours in common. Assume that $C(\til{e}_0)$ is isomorphic to some $C(x)$ with $x \in \E''_n$. Let $A \in V(C(\til{e}_0))$ be mapped to a special element, and $B \in V(C(\til{e}_0))$ be mapped to its nonspecial neighbour. The degree of $B$ equals $K = 2^n - 4 \geq 8$. Then, by Lemma~\ref{lemma:find-special-elements}, $A$ and $B$ have $\frac{K - 4}{2} \geq 2$ common neighbours, a contradiction.

Hence, if $\chi = -\gamma_{n-1}\gamma_n = -1$, then we can always find $C(\til{e}_0)$ as the only connected component with $\chi = -1$ but without special elements. Conversely, if such a component exists, then $\gamma_{n-1}\gamma_n = 1$. Otherwise, $\gamma_{n-1}\gamma_n = -1$, and it remains to find the component which is isomorphic to the corresponding $C(\til{e}_0)$ from Fig.~\ref{figure:special-components}. We can now determine $\gamma_{n-1}\gamma_n$, and thus we also know $\gamma_{n-1}$.
\end{proof}

\subsection{Distinguishing algebras for $n \leq 3$} \label{subsection:distinguish-low-dimensional}

We have already seen in Lemma~\ref{lemma:find-special-components} that we can find $C(\til{e}_0)$ and determine $\gamma_{n-1}$ for $n \geq 4$ only.  Corollary~\ref{corollary:retrieve-n-and-chi} provides a result on the number of connected components for $n \geq 3$ only, so we can find $\gamma_n$ for $n \geq 3$. Corollary~\ref{corollary:zero-divisors-criterion} which describes the vertices of $\Gamma_e(\A_{n+1})$ also has some exceptions for $n \leq 2$. We will see soon that we cannot even distinguish the connected components of types~\ref{item:component-type-3} and~\ref{item:component-type-4} if $n \leq 3$. Due to these reasons we should separately describe the cases when $n = 1$, $n = 2$, and $n = 3$. As we have shown in Subsection~\ref{subsection:low-dimensional-examples}, for $n = 0$ the graph $\Gamma_e(\A_{n+1})$ always has an empty set of vertices.

We will use mainly Proposition~\ref{proposition:chi-value} and the following proposition.

\begin{proposition}[{\cite[Proposition~5.10]{our_anticomm}}] \label{proposition:basis-norms}
Let $n \geq 1$. If $\gamma_k = -1$ for all $k = 0, \dots, n-1$, then $n(e^{(n)}_j) = 1$ for all $j = 0, \dots, 2^n -1$. Otherwise, $2^{n-1}$ elements of $\E_n$ have norm equal to $1$, and the other $2^{n-1}$ elements have norm equal to $-1$, whereas $n(e_0) = 1$ always holds. 
\end{proposition}

\begin{corollary} \label{corollary:components-chi}
The values of $\chi$ for $C(x)$, $x \in \E_n$, are as follows:
\begin{enumerate}
    \item If $\gamma_k = -1$ for all $k = 0, \dots, n-1$, then $\chi(C(x)) = \gamma_n$ for all $x \in \E_n$.
    \item Otherwise, $\chi(C(x))$ equals $1$ for $2^{n-1}$ elements $x \in \E_n$, and $-1$ for the other $2^{n-1}$ elements. Moreover, we always have $\chi(C(e_0)) = \gamma_n$.
\end{enumerate}
\end{corollary}

\begin{proof}
By Proposition~\ref{proposition:chi-value}, $\chi(C(x)) = \gamma_n n(x)$. It remains to apply Proposition~\ref{proposition:basis-norms}.
\end{proof}

\begin{lemma} \label{lemma:low-dimensional-components}
For $1 \leq n \leq 3$ we can find the explicit form of $C(x)$, $x \in \E_n$, depending on the value of its $\chi$.
\begin{enumerate}[label={$n = \arabic*:$}, leftmargin=*]
    \item If $\chi = -1$, then $C(x)$ has an empty set of vertices. If $\chi = 1$, then $C(x)$ consists of two disconnected vertices.
    \item If $\chi = -1$, then $C(x)$ has an empty set of vertices. If $\chi = 1$, then $C(x)$ is a hexagon.
    \item Assume first that $x \in \E'_n$. If $\chi = -1$, then $C(x)$ is a double hexagon, see Fig.~\ref{figure:double-hexagon}. If $\chi = 1$, then $C(x)$ is a bundle of hexagons, see Fig.~\ref{figure:bundle}.
    
    Consider now the case when $x = e_0$. If $\chi = -1$, then $C(e_0)$ has an empty set of vertices. If $\chi = 1$, then $C(e_0)$ is an almost complete $(7,7)$-bipartite graph, see $C(e_0)$ of $\Gamma_e(\hat{\mathbb{S}})$ in Fig.~\ref{figure:H4-graph}.
\end{enumerate}
\end{lemma}

\begin{proof}
It follows from Corollary~\ref{corollary:zero-divisors-criterion} that there are no zero divisors with $\chi = -1$ for $n \leq 2$, while any element with $\chi = 1$ is always a zero divisor. The rest follows immediately from Lemma~\ref{lemma:special-components} and Theorem~\ref{theorem:nonspecial-components}. We use here the fact that:
\begin{enumerate}[label={$n = \arabic*:$}, leftmargin=*]
    \item $\E_n = \{ e_0, \til{e}_0 \}$;
    \item $\E_n = \{ e_0, \til{e}_0, a, \til{a} \}$ up to signs for any $a \in \E''_n$;
    \item $\E_n = \{ e_0, \til{e}_0, a, \til{a}, b, \til{b}, ab, \til{ab} \}$ up to signs if $a,b \in \E''_n$ and $a,\til{a},b,\til{b}$ are linearly independent.
\end{enumerate}
Note that there is no difference between components of different types, except for~\ref{item:component-type-1} when $n = 3$, and their structure depends only on the value of $\chi$.
\end{proof}

\begin{theorem} \label{theorem:low-dimensional-graphs}
If $1 \leq n \leq 3$, then $\Gamma_e(\A_{n+1})$ is isomorphic to one of the following graphs.
\begin{enumerate}[label={$n = \arabic*:$}, leftmargin=*]
    \item
    \begin{itemize}
        \item If $\gamma_0 = \gamma_1 = -1$, then $\Gamma_e(\A_{n+1})$ has an empty set of vertices.
        \item If $\gamma_0 = 1$, then $\Gamma_e(\A_{n+1})$ consists of two disconnected vertices.
        \item If $\gamma_0 = -1$ and $\gamma_1 = 1$, then $\Gamma_e(\A_{n+1})$ consists of four disconnected vertices, see Fig.~\ref{figure:H2-graph}.
    \end{itemize}
    \item
    \begin{itemize}
        \item If $\gamma_0 = \gamma_1 = \gamma_2 = -1$, then $\Gamma_e(\A_{n+1})$ has an empty set of vertices.
        \item If $\gamma_0 = 1$ or $\gamma_1 = 1$, then $\Gamma_e(\A_{n+1})$ consists of two hexagons.
        \item If $\gamma_0 = \gamma_1 = -1$ and $\gamma_2 = 1$, then $\Gamma_e(\A_{n+1})$ consists of four hexagons, see Fig.~\ref{figure:H3-graph}.
    \end{itemize}
    \item
    \begin{itemize}
        \item If $\gamma_0 = \gamma_1 = \gamma_2 = \gamma_3 = -1$, then $\Gamma_e(\A_{n+1})$ consists of seven double hexagons, see Fig.~\ref{figure:M4-graph}.
        \item If at least one of $\gamma_0, \gamma_1, \gamma_2$ equals $1$ and $\gamma_3 = -1$, then $\Gamma_e(\A_{n+1})$ consists of three double hexagons and four bundles of hexagons.
        \item If at least one of $\gamma_0, \gamma_1, \gamma_2$ equals $1$ and $\gamma_3 = 1$, then $\Gamma_e(\A_{n+1})$ consists of four double hexagons, three bundles of hexagons and one almost complete $(7,7)$-bipartite graph.
        \item If $\gamma_0 = \gamma_1 = \gamma_2 = -1$ and $\gamma_3 = 1$, then $\Gamma_e(\A_{n+1})$ consists of seven bundles of hexagons and one almost complete $(7,7)$-bipartite graph, see Fig.~\ref{figure:H4-graph}.
    \end{itemize}
\end{enumerate}
\end{theorem}

\begin{proof}
Follows immediately from Corollary~\ref{corollary:components-chi} and Lemma~\ref{lemma:low-dimensional-components}. For $n \leq 2$ we need not consider $C(e_0)$ independently, and all $C(x)$ with $\chi = -1$ have empty sets of vertices. As for $n = 3$, $C(e_0)$ differs from other components with the same $\chi$, and it is the only component which can vanish, namely, for $\gamma_n = -1$. This is due to Corollary~\ref{corollary:retrieve-n-and-chi} which takes place exactly for $n \geq 3$. For this reason there are four cases for $n = 3$, instead of three cases for $n = 1$ and $n = 2$.
\end{proof}

\begin{corollary} \label{corollary:low-dimensional-isomorphic}
Let $1 \leq n, m \leq 3$, $\gamma = \{ \gamma_0, \dots, \gamma_n \}$ and $\lambda = \{ \lambda_0, \dots, \lambda_m \}$ be two sequences of parameters. Assume that $\Gamma_e(\A^{\gamma}_{n+1})$ and $\Gamma_e(\A^{\lambda}_{m+1})$ both have nonempty sets of vertices and are isomorphic. Then $\A^{\gamma}_{n+1}$ and $\A^{\lambda}_{m+1}$ are also isomorphic.
\end{corollary}

\begin{proof}
By Theorem~\ref{theorem:low-dimensional-graphs}, the size of every connected component equals one for $n = 1$, six for $n = 2$, and twelve or fourteen for $n = 3$. Since $\Gamma_e(\A^{\gamma}_{n+1})$ and $\Gamma_e(\A^{\lambda}_{m+1})$ have nonempty sets of vertices and are isomorphic, we immediately obtain $n = m$. Consider now three cases:
\begin{enumerate}[label={$n = \arabic*:$}, leftmargin=*]
    \item The only two cases when distinct algebras have isomorphic graphs are $\gamma_0 = 1$, $\gamma_1 = -1$ and $\gamma_0 = 1$, $\gamma_1 = 1$. By Proposition~\ref{proposition:split-algebras}, both algebras are isomorphic to $\Hyp_2$.
    \item Distinct algebras have isomorphic graphs when at least one of the parameters $\gamma_0$ and $\gamma_1$ equals $1$. However, by Proposition~\ref{proposition:split-algebras}, we obtain algebras isomorphic to $\Hyp_3$ in all six cases.
    \item According to Corollary~\ref{corollary:retrieve-n-and-chi}, we have seven connected components if $\gamma_3 = -1$, and eight connected components if $\gamma_3 = 1$. Therefore, $\gamma_3$ is uniquely determined.
    
    Let, for example, $\gamma_3 = -1$. Distinct algebras have isomorphic graphs if and only if at least one of $\gamma_0, \gamma_1, \gamma_2$ equals $1$. However, in all seven cases the algebra $\A_3$ is isomorphic to $\Hyp_3$. Hence $\A_4 = \A_3 \{ -1 \}$ is isomorphic to $\Hyp_3 \{ -1 \}$ (we apply the isomorphism componentwise).
    
    The case $\gamma_3 = 1$ is entirely similar. \qedhere
\end{enumerate}
\end{proof}

We will see in Theorem~\ref{theorem:isomorphic-graphs} that for $n = 3$ the converse to Corollary~\ref{corollary:low-dimensional-isomorphic} holds, that is, if algebras are isomorphic, then their graphs are also isomorphic.

However, if $n \leq 2$, then the converse to Corollary~\ref{corollary:low-dimensional-isomorphic} is not true. By Proposition~\ref{proposition:split-algebras}, for $n \leq 2$ all real Cayley--Dickson algebras $\A_{n+1}$ are isomorphic either to algebras of the main sequence $\Main_{n+1}$ or to split-algebras $\Hyp_{n+1}$. Particularly, $\A_2\{1,-1\} \cong \A_2\{-1,1\} = \Hyp_2$ but, by Theorem~\ref{theorem:low-dimensional-graphs}, their graphs are nonisomorphic. Similarly, the graphs of isomorphic algebras $\A_3\{1,-1,-1\}$ and $\A_3\{-1,-1,1\} = \Hyp_3$ are also nonisomorphic. We will discuss in Section~\ref{section:another-approach} how we can cope with this problem.

\subsection{Building graphs recursively}

We have demonstrated in Subsection~\ref{subsection:finding-last-two} how $\gamma_{n-1}$ and $\gamma_n$ are determined from $\Gamma_e(\A_{n+1})$ for $n \geq 4$. We now obtain from $\Gamma_e(\A_{n+1})$, $n \geq 4$, either $\Gamma_e(\A^{\circ}_n)$ or $\Gamma_e(\A^{\bullet}_n)$. This helps us to find all other values of $\gamma_k$, $k = 0, \dots, n$, recursively.

Theorem~\ref{theorem:nonspecial-components} implies that for $x \in \E'_{n-1}$ the component $C(x)$ should be reduced to the connected component $C^{\circ}(x)$ of $\Gamma_e(\A^{\circ}_n)$, while $C(\til{x})$ should be reduced to the connected component $C^{\bullet}(x)$ of $\Gamma_e(\A^{\bullet}_n)$. We first see how the value of $\chi$ is changed under this transformation.

\begin{lemma} \label{lemma:chi-transformation}
Let $x \in \E'_{n-1}$.
\begin{enumerate}
    \item If $C(x)$ is transformed into $C^{\circ}(x)$, then $\chi$ is preserved.
    \item If $C(\til{x})$ is transformed into $C^{\bullet}(x)$, then $\chi$ changes sign.
\end{enumerate}
\end{lemma}

\begin{proof}
We use Proposition~\ref{proposition:chi-value} to find the values of $\chi$.
\begin{enumerate}
    \item We have $\chi(C(x)) = \gamma_n n(x)$. Since $x \in \A_{n-1}$, the norm of $x$ can also be considered in $\A_{n-1}$. Recall now that $\A_n^{\circ} = \A_{n-1} \{ \gamma_n \}$. Clearly, $\chi(C^{\circ}(x)) = \gamma_n n(x)$, too.
    \item We have $\chi(C(\til{x})) = \gamma_n n(\til{x}) = \gamma_n n(x)n(\til{e}_0) = -\gamma_{n-1} \gamma_n n(x)$ in $\Gamma_e(\A_{n+1})$. Since $\A_n^{\bullet} = \A_{n-1} \{ \gamma_{n-1}\gamma_n \}$, $\chi(C^{\bullet}(x)) = \gamma_{n-1} \gamma_n n(x)$. Therefore, $\chi$ changes sign. \qedhere
\end{enumerate}
\end{proof}

\begin{lemma} \label{lemma:nonspecial-components-transformation}
Let $x \in \E'_{n-1}$, $n \geq 4$. Then Fig.~\ref{figure:nonspecial-components-transformation} shows how $C(x)$ and $C(\til{x})$ are reduced to $C^{\circ}(x)$ and $C^{\bullet}(x)$. The bold lines denote those edges and vertices which are preserved, while all others are deleted.
\begin{enumerate}
    \item $C(x)$ is transformed into $C^{\circ}(x)$ as follows:
    \begin{enumerate}[label=(\alph*)]
        \item If $\chi = -1$, then we find some special element and mark all its neighbours. Then we delete all elements except for the marked. In particular, we delete all special elements. In other words, we take the subgraph of $C(x)$ on the marked vertices.
        \item If $\chi = 1$, then we find two opposite special elements and delete two isolated columns between them which correspond to $(\til{x}, \pm \til{e}_0)$. Then we divide all other columns into pairs as follows: there are no edges between two columns of a pair, however, they have the same sets of neighbour columns. Finally, we keep only one column from each pair.
    \end{enumerate}
        \newpage
\begin{figure}[H]
\centering
\includegraphics[width=\linewidth]{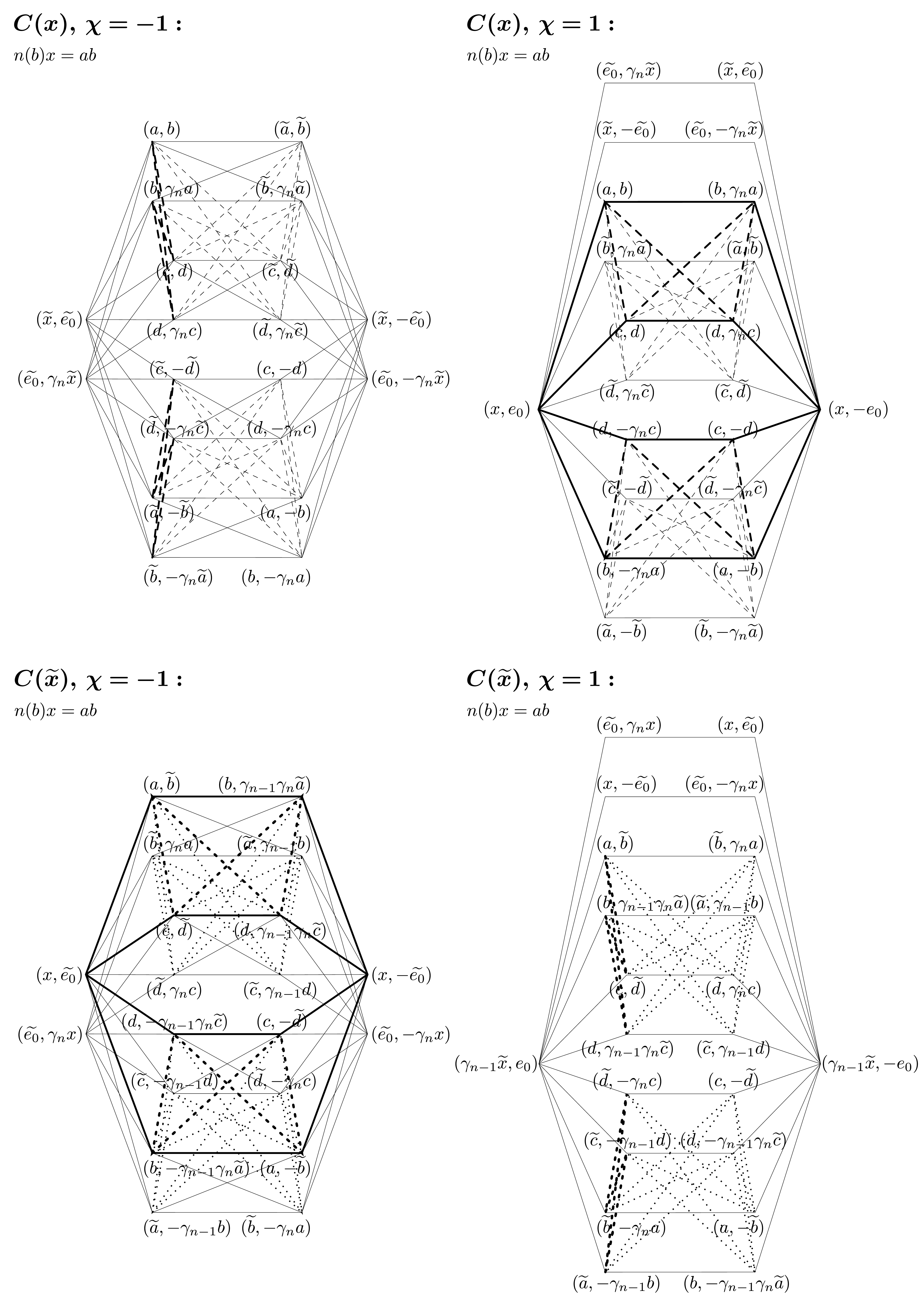}
\caption{\label{figure:nonspecial-components-transformation} Recursive transformation of connected components of types~\ref{item:component-type-3} and~\ref{item:component-type-4}.}
\end{figure}

    \item $C(\til{x})$ is transformed into $C^{\bullet}(x)$ as follows:
    \begin{enumerate}[label=(\alph*)]
        \item If $\chi = -1$, then we divide all elements into pairs: two elements of a pair are not adjacent but they have the same sets of neighbours. Then we keep only one element from each pair.
        \item If $\chi = 1$, then we find two opposite special elements and delete two isolated columns between them which correspond to $(\til{x}, \pm \til{e}_0)$. Then we mark all remaining neighbours of some special element. Finally, we delete all elements except for the marked.
    \end{enumerate}
\end{enumerate}
\end{lemma}

\begin{proof}
We will use Theorem~\ref{theorem:nonspecial-components} which describes $C(x)$ and $C(\til{x})$ for various values of~$\chi$. Recall that, since $n \geq 4$, special elements can be found by Lemma~\ref{lemma:find-special-elements}.
\begin{enumerate}
    \item We first show how $C(x)$ is transformed into $C^{\circ}(x)$. By Lemma~\ref{lemma:chi-transformation}, $\chi$ is preserved.
    
    \begin{enumerate}[label=(\alph*)]
        \item Let $\chi = -1$. By Corollary~\ref{corollary:retrieve-n-and-chi}, $C(x)$ contains $2^{n+1} - 4$ vertices. Then $\chi(C^{\circ}(x))$ also equals $-1$, and $C^{\circ}(x)$ contains $2^n - 4$ vertices. There are no elements of the form $(x, \pm e_0)$ in both cases. Hence $C^{\circ}(x)$ consists only of the vertices $(a,b)$ for $a \in \E'_{n-1}$, $b \in \pm \E'_{n-1}$ and $ab = \pm x$. It follows from Lemma~\ref{lemma:distinct-nonspecial-element} that the subgraph of $C(x)$ on the elements of the form $(a,b)$ is indeed isomorphic to $C^{\circ}(x)$. Note also that, by Theorem~\ref{theorem:common-nonspecial-element}, the element $(a,b)$ is not adjacent to $(a,-b)$ and $(b,\pm a)$ both in $C(x)$ and in $C^{\circ}(x)$. Therefore, we want to delete special elements $(\til{x}, \pm \til{e}_0)$ and $(\til{e}_0, \pm \gamma_n \til{x})$ together with elements of the form $(\til{a},\til{b})$.
        
        However, we cannot distinguish between $(a,b)$ and $(\til{a},\til{b})$ in the graph, except for the fact that they are connected to distinct pairs of special vertices. Moreover, half of the elements $(a,b)$ are connected to one pair of special elements, and another half to the other, and we do not know how they are divided in half. Therefore, we cannot actually find the subgraph of $C(x)$ on the elements of the form $(a,b)$ but we can find one isomorphic to it. Indeed, if we take all neighbours of some pair of special elements, then we always choose exactly one element from $(a,b)$ and $(\til{a},\til{b})$. Since $(a,b)$ and $(\til{a},\til{b})$ have the same sets of nonspecial neighbours, we obtain a graph isomorphic to $C^{\circ}(x)$.
        
        \item Let $\chi = 1$. Then $C(x)$ contains $2^{n+1} - 2$ vertices, $\chi(C^{\circ}(x))$ also equals $1$, and $C^{\circ}(x)$ contains $2^n - 2$ vertices. There are elements of the form $(x, \pm e_0)$ in both cases. The subgraph of $C(x)$ on the elements of the form $(a,b)$ and $(x,\pm e_0)$ is isomorphic to $C^{\circ}(x)$. Indeed, for $a,c \in \E'_{n-1}$, $b,d \in \pm \E_{n-1}$ the elements $(a,b)$ and $(c,d)$ are orthogonal in $\A_{n+1}$ if and only if $(a,b)$ and $(c,d)$ are orthogonal in $\A^{\circ}_n$. We consider three cases independently: if $b = \pm e_0$ or $d = \pm e_0$, then we use Theorem~\ref{theorem:e_0-special-element}; if $e_0,a,b,c,d$ are linearly independent, then we use Lemma~\ref{lemma:distinct-nonspecial-element}; if $c = a$ or $d = \pm a$, then we use Theorem~\ref{theorem:common-nonspecial-element}.
        
        Therefore, we want to delete the elements $(\til{x}, \pm \til{e}_0)$ and $(\til{e}_0, \pm \gamma_n \til{x})$ together with elements of the form $(\til{a},\til{b})$ for $a \in \E'_{n-1}$, $b \in \pm \E'_{n-1}$. We first find special elements by Lemma~\ref{lemma:find-special-elements}. Then the elements $(\til{x}, \pm \til{e}_0)$ and $(\til{e}_0, \pm \gamma_n \til{x})$ can be easily found, since they belong to two columns between $(x,e_0)$ and $(x,-e_0)$ which are not adjacent to any other columns. We delete these two columns and then work with the rest of graph.
        
        We are now to delete the elements of the form $(\til{a}, \til{b})$. Note that we cannot distinguish between the column which contains $(a,b)$ and $(b, \gamma_n a)$, and the column which contains $(\til{a},\til{b})$ and $(\til{b},\gamma_n \til{a})$. Clearly, we need only one of them. Then we can divide all columns into pairs as follows: there are no edges between two columns of a pair, however, they have the same sets of neighbour columns. Finally, we keep only one column from each pair and thus obtain a graph which is isomorphic to $C^{\circ}(x)$.
    \end{enumerate}
    
    \item Next we show how $C(\til{x})$ is transformed into $C^{\bullet}(x)$. By Lemma~\ref{lemma:chi-transformation}, $\chi$ changes sign.
    
    \begin{enumerate}[label=(\alph*)]
        \item Let $\chi = -1$. Then $C(\til{x})$ contains $2^{n+1} - 4$ vertices, $\chi(C^{\bullet}(x))$ equals $1$, and $C^{\bullet}(x)$ contains $2^n - 2$ vertices. There are no elements $(\til{x},\pm e_0)$ in $C(\til{x})$, but there are elements $(x, \pm e_0)$ in $C^{\bullet}(x)$. Hence $C^{\bullet}(x)$ consists of $(x,\pm e_0)$ and the vertices $(a,b)$ for $a \in \E'_{n-1}$, $b \in \pm \E'_{n-1}$ and $ab = \pm x$. We claim that the subgraph of $C(\til{x})$ on all elements of the form $(c,\til{d})$ for $c \in \E'_{n-1}$, $d \in \pm \E_{n-1}$, is isomorphic to $C^{\bullet}(x)$. The required isomorphism maps $(c,\til{d})$ to $(c,d)$.
        
        Indeed, for $a,c \in \E'_{n-1}$, $b,d \in \pm \E_{n-1}$ the elements $(a,\til{b})$ and $(c,\til{d})$ are orthogonal in $\A_{n+1}$ if and only if $(a,b)$ and $(c,d)$ are orthogonal in $\A^{\bullet}_n$. We consider three cases independently: if $b = \pm e_0$ or $d = \pm e_0$, then we use Theorem~\ref{theorem:e_0-special-element}; if $e_0,a,b,c,d$ are linearly independent, then we use Lemma~\ref{lemma:distinct-nonspecial-element}; if $c = a$ or $d = \pm a$, then we use Theorem~\ref{theorem:common-nonspecial-element}. We also use here the fact that the last $\gamma$ for $\A^{\bullet}_n$ equals $\gamma_{n-1}\gamma_n$. 
        
        We now want to determine the subgraph of $C(\til{x})$ on the elements of the form $(a,\til{b})$ for $a \in \E'_{n-1}$, $b \in \pm \E_{n-1}$. However, we cannot distinguish between $(a,\til{b})$ and $(\til{b},\gamma_n a)$ in the graph. Hence we divide all elements into pairs such that two elements of a pair are not adjacent but they have the same sets of neighbours. Then $(a,\til{b})$ and $(\til{b},\gamma_n a)$ will get into the same pair, up to a graph isomorphism. If we keep exactly one element from each pair, then we will obtain a graph isomorphic to $C^{\bullet}(x)$.
        
        \item Let $\chi = 1$. Then $C(\til{x})$ contains $2^{n+1} - 2$ vertices, $\chi(C^{\bullet}(x))$ equals $-1$, and $C^{\bullet}(x)$ contains $2^n - 4$ vertices. There are elements $(\til{x}, \pm e_0)$ in $C(\til{x})$, but there are no elements $(x, \pm e_0)$ in $C^{\bullet}(x)$. We claim that the subgraph of $C(\til{x})$ on all elements of the form $(a,\til{b})$ for $a \in \E'_{n-1}$, $b \in \pm \E'_{n-1}$, is isomorphic to $C^{\bullet}(x)$. The required isomorphism maps $(a,\til{b})$ to $(a,b)$.
        
        Indeed, for $a,c \in \E'_{n-1}$, $b,d \in \pm \E'_{n-1}$ the elements $(a,\til{b})$ and $(c,\til{d})$ are orthogonal in $\A_{n+1}$ if and only if $(a,b)$ and $(c,d)$ are orthogonal in $\A^{\bullet}_n$. If $a,b,c,d$ are linearly independent, then we use Lemma~\ref{lemma:distinct-nonspecial-element}, and if $c = a$ or $d = \pm a$, then we use Theorem~\ref{theorem:common-nonspecial-element}. We recall here that the last $\gamma$ for $\A^{\bullet}_n$ equals $\gamma_{n-1}\gamma_n$.
        
        To construct this subgraph of $C(\til{x})$, we first find special elements and delete the elements $(x, \pm \til{e}_0)$ and $(\til{e}_0, \pm \gamma_n x)$ as for $C(x)$ with $\chi = 1$. Then we note that $(a,\til{b})$ and $(\til{b},\gamma_n a)$ are adjacent and have the same neighbours, except for special elements. Therefore, we can take all neighbours of some special element, as for $C(x)$ with $\chi = -1$, and obtain a graph which is isomorphic to $C^{\bullet}(x)$. \qedhere
    \end{enumerate}
\end{enumerate}
\end{proof}

\begin{remark}
We do not need labels of elements on the vertices of $C(x)$ and $C(\til{x})$ to apply Lemma~\ref{lemma:nonspecial-components-transformation}. Particularly, since $n \geq 4$, special elements can be determined by Lemma~\ref{lemma:find-special-elements}.
\end{remark}

We can now reduce $C(x)$ and $C(\til{x})$ to $C^{\circ}(x)$ and $C^{\bullet}(x)$. However, we cannot yet distinguish between $C(x)$ and $C(\til{x})$, so we do not know which procedure is to be applied to some specific component. The following lemma provides a method to recognize it.

\begin{lemma} \label{lemma:distinguish-types-3-and-4}
Consider $C(x)$ for $x \in \E''_n$ and $n \geq 4$.
\begin{enumerate}
    \item If its $\chi = 1$, then we consider the number of isolated columns between special elements. If there are at least six of them, then $x \in \E'_{n-1}$. Otherwise, there are only two of them, and $x = \til{y}$ for some $y \in \E'_{n-1}$.
    \item If its $\chi = -1$, then we consider the number of isolated double columns between pairs of special elements. If there are no isolated double columns, then $x \in \E'_{n-1}$. Otherwise, there are at least two of them, and $x = \til{y}$ for some $y \in \E'_{n-1}$.
\end{enumerate}
\end{lemma}

\begin{proof}
We will use Lemma~\ref{lemma:nonspecial-components-transformation} for the explicit form of transformations.
\begin{enumerate}
    \item Let $\chi = 1$. If $x \in \E'_{n-1}$, then $C(x)$ is reduced to $C^{\circ}(x)$ with the same $\chi$ as in Fig.~\ref{figure:nonspecial-components-transformation}. $C(x)$ has two isolated columns which correspond to $(\til{x}, \pm \til{e}_0)$ and $(\til{e}_0, \pm \gamma_n \til{x})$, and these columns are deleted under this transformation. $C^{\circ}(x)$ has $\chi = 1$ and thus also has two isolated columns which correspond to four isolated columns in $C(x)$. Therefore, $C(x)$ has at least six isolated columns.
    
    If $x = \til{y}$ for some $y \in \E'_{n-1}$, then $C(\til{y})$ is reduced to $C^{\bullet}(y)$ with the opposite $\chi$. Since $n \geq 4$, $C^{\bullet}(y)$ is connected by Corollary~\ref{corollary:diameters}. Hence the only two isolated columns in $C(\til{y})$ correspond to $(y, \pm \til{e}_0)$ and $(\til{e}_0, \pm \gamma_n y)$.
    
    \item Let now $\chi = -1$. If $x = \til{y}$ for some $y \in \E'_{n-1}$, then $C(\til{y})$ is reduced to $C^{\bullet}(y)$ with the opposite $\chi$. $C^{\bullet}(y)$ has $\chi = 1$ and thus has at least two isolated columns. Hence $C(\til{y})$ has at least two isolated double columns.
    
    If $x \in \E'_{n-1}$, then $C(x)$ is reduced to $C^{\circ}(x)$ with the same $\chi$. Since $n \geq 4$, $C^{\circ}(x)$ is connected by Corollary~\ref{corollary:diameters}, so $C(x)$ has no isolated double columns. \qedhere
\end{enumerate}
\end{proof}

\begin{theorem} \label{theorem:construct-graphs-recursively}
Let $n \geq 4$, and we have $\Gamma_e(\A_{n+1})$ without labels of elements on its vertices. Then we can determine $\Gamma_e(\A^{\circ}_n)$ and $\Gamma_e(\A^{\bullet}_n)$.
\end{theorem}

\begin{proof}
By Lemma~\ref{lemma:find-special-components}, we can determine $\gamma_{n-1}$ and $\gamma_n$ and find $C(e_0)$ and $C(\til{e}_0)$, if nonempty. Then, by definition, we know the last $\gamma$ for both $\A^{\circ}_n$ and $\A^{\bullet}_n$. Hence we automatically know $C^{\circ}(e_0)$ and $C^{\bullet}(e_0)$ from Lemma~\ref{lemma:special-components}, and it is sufficient to construct $C^{\circ}(x)$ and $C^{\bullet}(x)$ for all $x \in \E'_{n-1}$.

Let now $C(x)$ be a connected component of $\Gamma_e(\A_{n+1})$ for some $x \in \E''_n$. By Corollary~\ref{corollary:retrieve-n-and-chi}, we know $\chi(C(x))$. By Lemma~\ref{lemma:find-special-elements}, we can find special elements of $C(x)$ and then use Lemma~\ref{lemma:distinguish-types-3-and-4} to say whether $x \in \E'_{n-1}$ or $x = \til{y}$ for some $y \in \E'_{n-1}$. Then we use Lemma~\ref{lemma:nonspecial-components-transformation} and obtain either $C^{\circ}(x)$ or $C^{\bullet}(y)$. Thus we get all connected components of $\Gamma_e(\A^{\circ}_n)$ and $\Gamma_e(\A^{\bullet}_n)$.
\end{proof}

\subsection{Isomorphic graphs and isomorphic algebras}

We will need the following lemma concerning isomorphisms of Cayley--Dickson algebras. We remark that in~\cite{eakin} it is formulated for an arbitrary field~$\mathbb{F}$, and we consider only the case when $\mathbb{F} = \mathbb{R}$.

\begin{lemma}[{\cite[Corollary~2.6]{eakin}}] \label{lemma:recursive-isomorphism}
Let $n \geq 3$, $\gamma = \{ \gamma_0, \dots, \gamma_n \}$ and $\lambda = \{ \lambda_0, \dots, \lambda_n \}$ be two sequences of parameters. Then $\A^{\gamma}_{n+1} \cong \A^{\lambda}_{n+1}$ if and only if $\gamma_n = \lambda_n$ and $\A^{\gamma}_n \cong \A^{\lambda}_n$.
\end{lemma}

\begin{corollary} \label{lemma:isomorphic-parameter-sequences}
Let $n \geq 0$, $\gamma = \{ \gamma_0, \dots, \gamma_n \}$ and $\lambda = \{ \lambda_0, \dots, \lambda_n \}$ be two sequences of parameters. Then $\A^{\gamma}_{n+1} \cong \A^{\lambda}_{n+1}$ if and only if the following conditions are satisfied:
\begin{enumerate}
    \item $\gamma_k = \lambda_k$ for all $k = 3, \dots, n$;
    \item there exists $1$ among $\gamma_0, \gamma_1, \gamma_2$ if and only if there exists $1$ among $\lambda_0, \lambda_1, \lambda_2$.
\end{enumerate}
\end{corollary}

\begin{proof}
We will use the induction on $n$.
\begin{enumerate}
    \item If $0 \leq n \leq 2$, then the statement follows immediately from Proposition~\ref{proposition:split-algebras}, since for $n \leq 2$ the algebra $\A_{n+1}$ is isomorphic either to $\Main_{n+1}$ (all parameters are equal to~$-1$) or to $\Hyp_{n+1}$ (at least one of the parameters equals $1$).
    \item Let now $n \geq 3$. By Lemma~\ref{lemma:recursive-isomorphism}, $\A^{\gamma}_{n+1} \cong \A^{\lambda}_{n+1}$ if and only if $\gamma_n = \lambda_n$ and $\A^{\gamma}_n \cong \A^{\lambda}_n$. It remains to apply the induction hypothesis to $\A^{\gamma}_n$ and $\A^{\lambda}_n$. \qedhere
\end{enumerate}
\end{proof}

\begin{theorem} \label{theorem:isomorphic-algebras}
Let $n \geq 3$, and we have $\Gamma_e(\A_{n+1})$ without labels of elements on its vertices. Then we can find all $\gamma_k$, $k = 3, \dots, n$. The first three parameters $\gamma_0, \gamma_1, \gamma_2$ are determined up to isomorphism of $\A_3$.
\end{theorem}

\begin{proof}
We prove the theorem by induction on $n$.
\begin{enumerate}
    \item If $n = 3$, then we have our statement from Corollary~\ref{corollary:low-dimensional-isomorphic}.
    \item Let now $n \geq 4$. By Lemma~\ref{lemma:find-special-components}, we can find $\gamma_{n-1}$ and $\gamma_n$. Then we use Theorem~\ref{theorem:construct-graphs-recursively} to construct $\Gamma_e(\A^{\circ}_n)$.  By the induction hypothesis, from $\Gamma_e(\A^{\circ}_n)$ we can find all $\gamma_k$, $k = 3, \dots, n-2$, and $\gamma_0, \gamma_1, \gamma_2$ are determined up to isomorphism of $\A_3$. \qedhere
\end{enumerate}
\end{proof}

\begin{corollary} \label{corollary:isomorphic-algebras}
Let $n \geq 3$, $\gamma = \{ \gamma_0, \dots, \gamma_n \}$ and $\lambda = \{ \lambda_0, \dots, \lambda_n \}$ be two sequences of parameters. If $\Gamma_e(\A^{\gamma}_{n+1})$ and $\Gamma_e(\A^{\lambda}_{n+1})$ are isomorphic, then $\A^{\gamma}_{n+1}$ and $\A^{\lambda}_{n+1}$ are also isomorphic.
\end{corollary}

\begin{proof}
Follows immediately from Theorem~\ref{theorem:isomorphic-algebras} and Corollary~\ref{lemma:isomorphic-parameter-sequences}.
\end{proof}

For $n \geq 3$ the converse to Corollary~\ref{corollary:isomorphic-algebras} is also true.

\begin{theorem} \label{theorem:isomorphic-graphs}
Let $n \geq 3$, $\gamma = \{ \gamma_0, \dots, \gamma_n \}$ and $\lambda = \{ \lambda_0, \dots, \lambda_n \}$ be two sequences of parameters. Assume that $\A^{\gamma}_{n+1}$ and $\A^{\lambda}_{n+1}$ are isomorphic. Then $\Gamma_e(\A^{\gamma}_{n+1})$ and $\Gamma_e(\A^{\lambda}_{n+1})$ are also isomorphic.
\end{theorem}

\begin{proof}
We will use the induction on $n$.
\begin{enumerate}
    \item If $n = 3$, then, by Corollary~\ref{lemma:isomorphic-parameter-sequences}, $\gamma_3 = \lambda_3$, and there exists $1$ among $\gamma_0, \gamma_1, \gamma_2$ if and only if there exists $1$ among $\lambda_0, \lambda_1, \lambda_2$. Then the explicit form of $\Gamma_e(\A^{\gamma}_{n+1})$ and $\Gamma_e(\A^{\lambda}_{n+1})$ in Theorem~\ref{theorem:low-dimensional-graphs} shows that these graphs are isomorphic.
    
    \item If $n \geq 4$, then it follows from Corollary~\ref{lemma:isomorphic-parameter-sequences} that $\gamma_k = \lambda_k$ for all $k = 3, \dots, n$, and there exists $1$ among $\gamma_0, \gamma_1, \gamma_2$ if and only if there exists $1$ among $\lambda_0, \lambda_1, \lambda_2$. Hence $\A^{\gamma,\circ}_n \cong \A^{\lambda,\circ}_n$ and $\A^{\gamma,\bullet}_n \cong \A^{\lambda,\bullet}_n$. By the inductive hypothesis, $\Gamma_e(\A^{\gamma,\circ}_n) \cong \Gamma_e(\A^{\lambda,\circ}_n)$ and $\Gamma_e(\A^{\gamma,\bullet}_n) \cong \Gamma_e(\A^{\lambda,\bullet}_n)$. Then we use Lemma~\ref{lemma:special-components} and Theorem~\ref{theorem:nonspecial-components} to construct $\Gamma_e(\A^{\gamma}_{n+1})$ and $\Gamma_e(\A^{\lambda}_{n+1})$. We now show that Theorem~\ref{theorem:nonspecial-components} can be applied correctly, and $\Gamma_e(\A^{\gamma}_{n+1}) \cong \Gamma_e(\A^{\lambda}_{n+1})$.
    
    Assume first that $n = 4$, and we want to construct $\Gamma_e(\A_5)$ from $\Gamma_e(\A^{\circ}_4)$ and $\Gamma_e(\A^{\bullet}_4)$. It follows from Theorem~\ref{theorem:low-dimensional-graphs} that any connected component of $\Gamma_e(\A^{\circ}_4)$ and $\Gamma_e(\A^{\bullet}_4)$, except for $C(e_0)$, is either a double hexagon or a bundle of hexagons. Under the transformation from Theorem~\ref{theorem:nonspecial-components} every edge of a double hexagon will become a dotted or a dashed edge in the next graph, since it connects doubly pure vertices with linearly independent components. Meanwhile, none of the edges of a bundle of hexagons will result in the next graph.
    
    The value of $\chi$ of the new connected component can be determined by Lemma~\ref{lemma:chi-transformation}. Thus every connected component of $\Gamma_e(\A^{\circ}_4)$ is transformed in a certain way which depends only on its $\chi$. The same holds for any connected component of $\Gamma_e(\A^{\bullet}_4)$. Let $\Gamma_e(\A^{\gamma,\circ}_4) \cong \Gamma_e(\A^{\lambda,\circ}_4)$ and $\Gamma_e(\A^{\gamma,\bullet}_4) \cong \Gamma_e(\A^{\lambda,\bullet}_4)$. Theorem~\ref{theorem:low-dimensional-graphs} transforms $\Gamma_e(\A^{\gamma,\circ}_4)$ and $\Gamma_e(\A^{\gamma,\bullet}_4)$ into $\Gamma_e(\A^{\gamma}_5)$ in the same way as it transforms $\Gamma_e(\A^{\lambda,\circ}_4)$ and $\Gamma_e(\A^{\lambda,\bullet}_4)$ into $\Gamma_e(\A^{\lambda}_5)$. Hence $\Gamma_e(\A^{\gamma}_5) \cong \Gamma_e(\A^{\lambda}_5)$.
    
    Let now $n \geq 5$. We again use Lemma~\ref{lemma:chi-transformation} for the transformation of $\chi$. Since we constructed $\Gamma_e(\A^{\gamma,\circ}_n)$ and $\Gamma_e(\A^{\lambda,\circ}_n)$ manually in the previous step, we know types of all their connected components and of all their vertices. Moreover, we can establish the complete relationship between them. Therefore, Theorem~\ref{theorem:nonspecial-components} can be applied to them immediately, and isomorphic connected components of $\Gamma_e(\A^{\gamma,\circ}_n)$ and $\Gamma_e(\A^{\lambda,\circ}_n)$ are transformed into isomorphic connected components of $\Gamma_e(\A^{\gamma}_{n+1})$ and $\Gamma_e(\A^{\lambda}_{n+1})$. The same holds for $\Gamma_e(\A^{\gamma,\bullet}_n)$ and $\Gamma_e(\A^{\lambda,\bullet}_n)$. Therefore, $\Gamma_e(\A^{\gamma}_{n+1}) \cong \Gamma_e(\A^{\lambda}_{n+1})$. \qedhere
\end{enumerate}
\end{proof}

\section{Another approach} \label{section:another-approach}

We have shown in Theorem~\ref{theorem:isomorphic-graphs} that for $n \geq 3$ isomorphic algebras $\A_{n+1}$ have isomorphic graphs. However, this is not true for $n \leq 2$, see Subsection~\ref{subsection:distinguish-low-dimensional}. The reason is that we consider elements of the form $(e^{(n)}_i,\pm e^{(n)}_j)$ but, for example, the isomorphism between $\A_3\{1,-1,-1\}$ and $\Hyp_3$ maps some basis elements from the first half to the second half, and vise versa.

This problem will be solved if we consider zero divisors of the form $e^{(n+1)}_i \pm e^{(n+1)}_j$, where $e^{(n+1)}_i$ and $e^{(n+1)}_j$ are distinct elements of $\E'_{n+1}$, $i < j$. Then we will have three types of elements:
\begin{enumerate}[label=(\roman*)]
    \item $(e^{(n)}_i, \pm e^{(n)}_j)$ with $0 < i$; \label{item:element-type-1}
    \item $(e^{(n)}_i \pm e^{(n)}_j, 0)$ with $0 < i < j$; \label{item:element-type-2}
    \item $(0, e^{(n)}_i \pm e^{(n)}_j)$ with $i < j$. \label{item:element-type-3}
\end{enumerate}

Since $(e^{(n)}_i, \pm e^{(n)}_j)$ is a doubly alternative element which satisfies condition~\eqref{equation:norm-condition}, one can infer from~\cite[Lemma~3.14]{our_orthographs1} that elements of type~\ref{item:element-type-1} are not orthogonal to elements of types~\ref{item:element-type-2} and~\ref{item:element-type-3}. Hence the connected components containing them will be the same as in $\Gamma_e(\A_{n+1})$.

As for elements of types~\ref{item:element-type-2} and~\ref{item:element-type-3}, their orthogonality condition is reduced to the nulling of a certain product in $\A_n$. For example, $(a,0)(0,b) = (0, ba) = 0$ in $\A_{n+1}$ if and only if $ba = 0$. We know that $(a,0)$ is pure if and only if $a$ is pure. However, $(0,b)$ is pure even for nonpure $b$. Let $\Re(a) = 0$, $\Re(b) \neq 0$, $ba = 0$. Then, by Proposition~\ref{proposition:orthogonality-condition}, $(a,0)$ is orthogonal to $(0,b)$ but $b$ is not orthogonal to $a$. Therefore, we have to consider not only $\Gamma_O(\A_n)$ but also $\Gamma_Z(\A_n)$.

Note that Lemma~\ref{lemma:basis-pairs-zero-divisors-relation} holds for pairs of zero divisors even if they are not orthogonal. Thence we conclude inductively that this new graph will consist of parts which correspond to all basis elements of $\A_n, \A_{n-1}, \dots, \A_1$, that is, 
\begin{align*}
& C\left(e^{\left(n\right)}_0\right)\!,\: C\left(e^{\left(n\right)}_1\right)\!,\: C\left(e^{\left(n\right)}_2\right)\!,\: \dots\!,\: C\left(e^{\left(n\right)}_{2^n-1}\right)\!,\:\\
& C\left(e^{\left(n-1\right)}_0\right)\!,\: C\left(e^{\left(n-1\right)}_1\right)\!,\: \dots\!,\: C\left(e^{\left(n-1\right)}_{2^{n-1}-1}\right)\!,\:\\
& \dots\\
& C\left(e^{\left(1\right)}_0\right)\!,\: C\left(e^{\left(1\right)}_1\right).
\end{align*}
Every two of them are not connected to each other. The first row corresponds to elements of type~\ref{item:element-type-1}, while the others are inherited from the graph of $\A_n$ and correspond to elements of types~\ref{item:element-type-2} and~\ref{item:element-type-3}.

\medskip

The author is grateful to her scientific advisor Professor Alexander E. Guterman for posing the
problem and fruitful discussions.

\newpage

\appendix

\section{Low-dimensional examples of $\Gamma_e(\A_{n+1})$}

\begin{figure}[H]
\centering
\includegraphics[width=\linewidth]{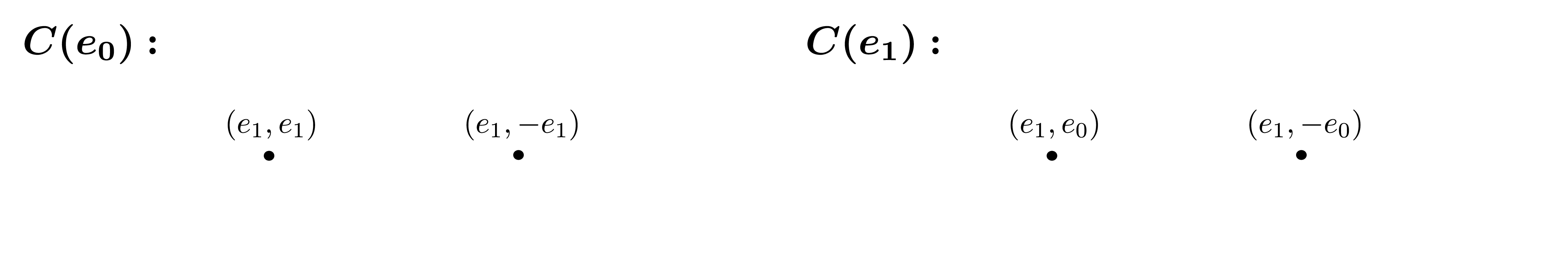}
\caption{\label{figure:H2-graph} The graph $\Gamma_e(\hat{\mathbb{H}})$.}
\end{figure}

\begin{figure}[H]
\centering
\includegraphics[width=\linewidth]{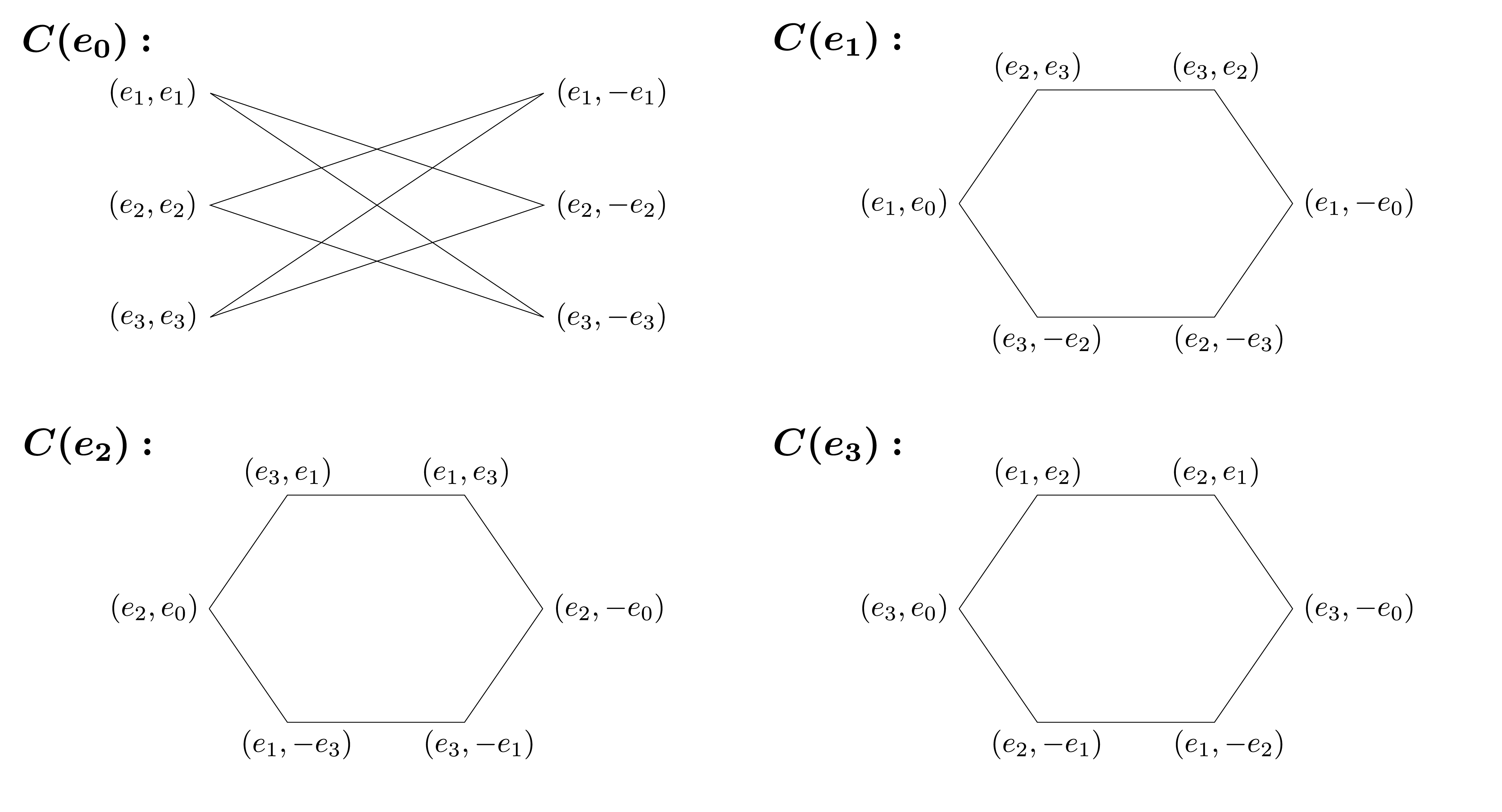}
\caption{\label{figure:H3-graph} The graph $\Gamma_e(\hat{\mathbb{O}})$.}
\end{figure}

\vfill

\newpage

\begin{figure}[H]
\centering
\includegraphics[width=\linewidth]{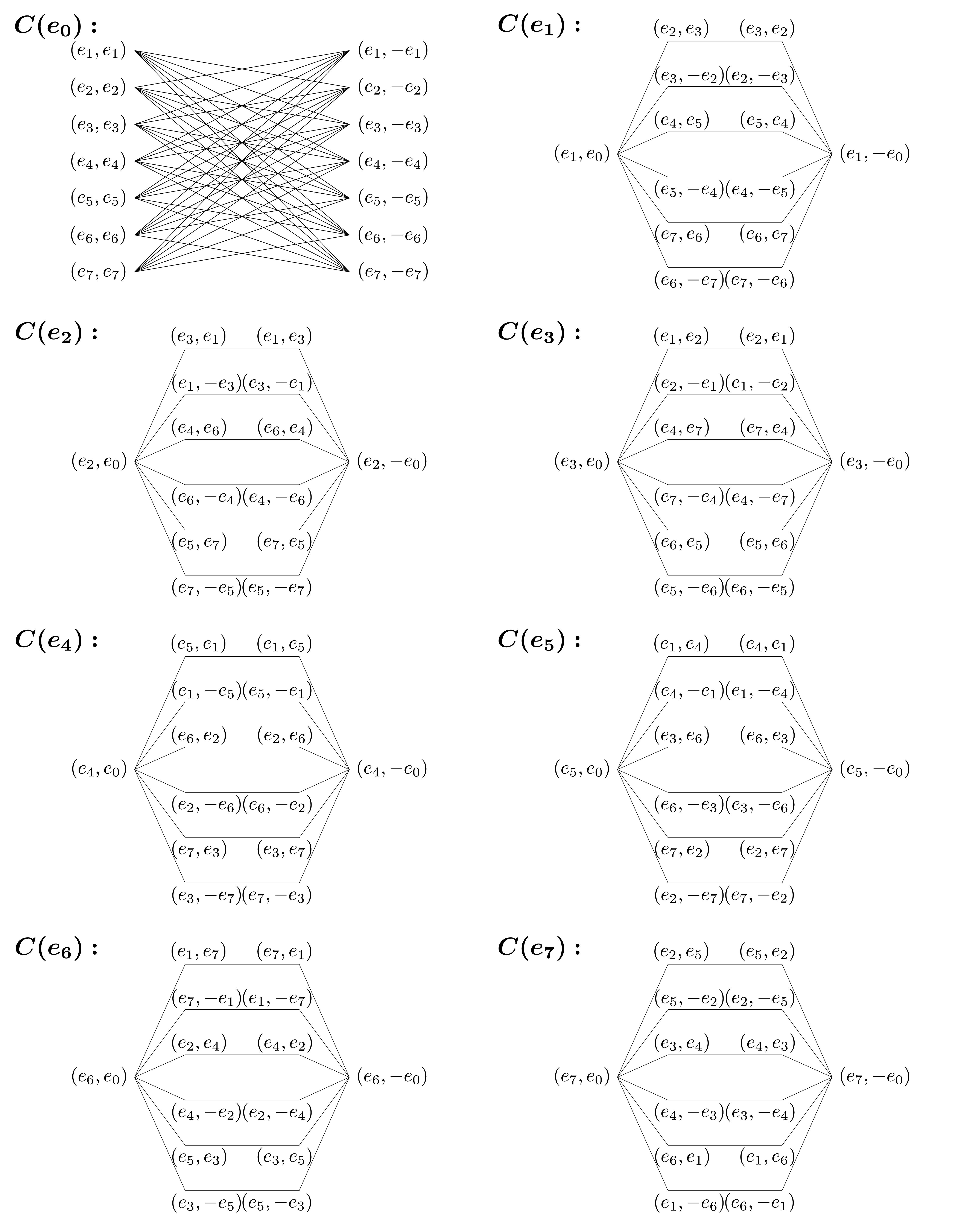}
\caption{\label{figure:H4-graph} The graph $\Gamma_e(\hat{\mathbb{S}})$.}
\end{figure}

\begin{figure}[H]
\centering
\includegraphics[width=\linewidth]{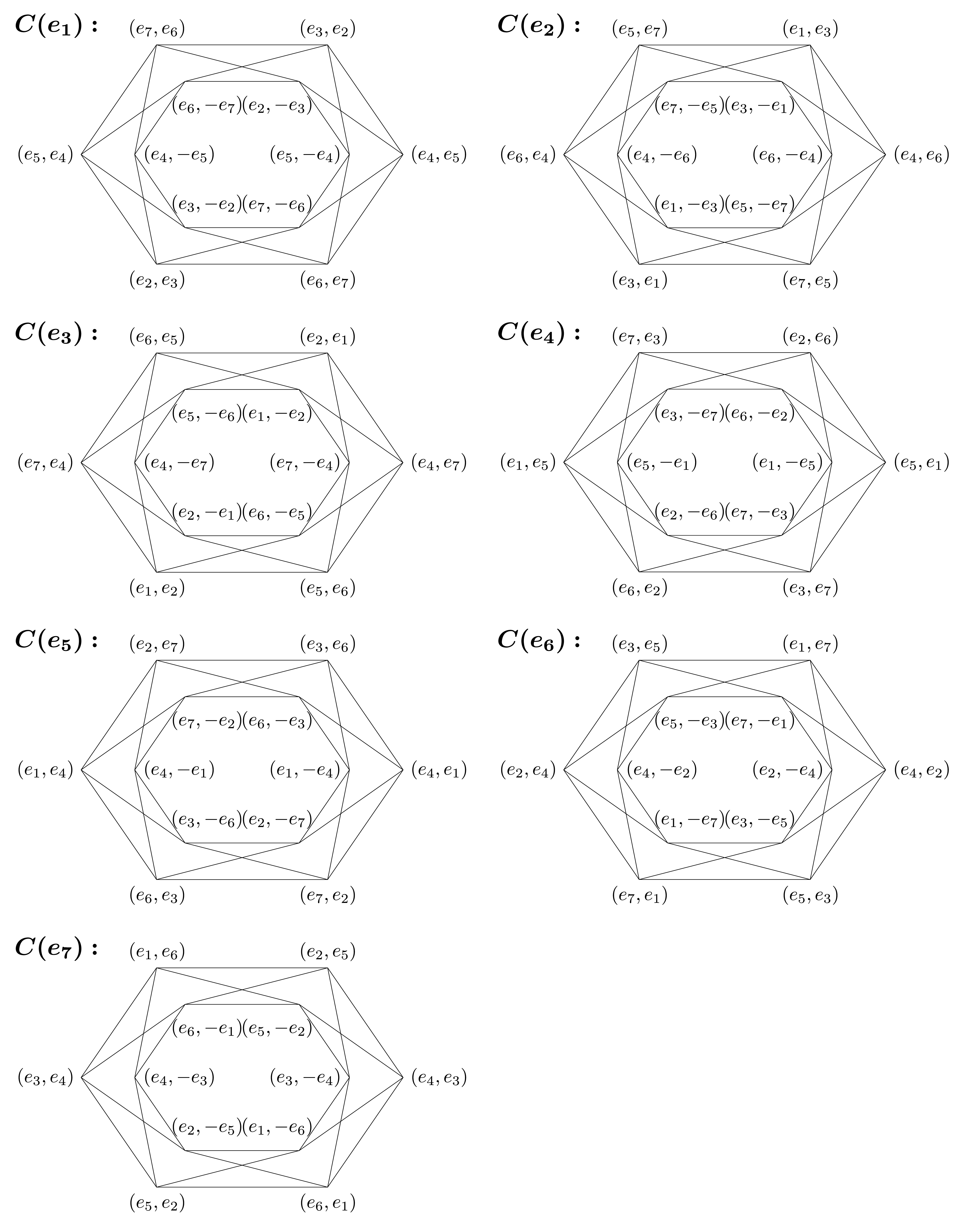}
\caption{\label{figure:M4-graph} The graph $\Gamma_e(\mathbb{S})$.}
\end{figure}

\end{document}